\documentclass[a4paper,12pt]{amsart}
\usepackage{amsmath,amssymb,inputenc,euscript,graphicx,psfrag,color}

\title[Wave equation on Damek--Ricci spaces]
{The wave equation on Damek--Ricci spaces}
 
\author{Jean--Philippe Anker}
 
\address{Universit\'e d'Orl\'eans \& CNRS,
F\'ed\'eration Denis Poisson (FR 2964) \& Laboratoire MAPMO (UMR 6628),
B\^atiment de math\'ematiques -- Route de Chartres,
B.P. 6759 -- 45067 Orl\'eans cedex 2 -- France}
 
\email{anker@univ-orleans.fr}

\author{Vittoria Pierfelice}

\address{Universit\'e d'Orl\'eans \& CNRS,
F\'ed\'eration Denis Poisson (FR 2964) \& Laboratoire MAPMO (UMR 6628),
B\^atiment de math\'ematiques -- Route de Chartres,
B.P. 6759 -- 45067 Orl\'eans cedex 2 -- France}
 
\email{vittoria.pierfelice@univ-orleans.fr}

\author{Maria Vallarino}

\address{Universita' di Milano -- Bicocca,
Dipartimento di Matematica e Applicazioni,
Via Cozzi 53 -- 20125 Milano -- Italia}

\email{maria.vallarino@unimib.it}

\thanks{This work was mostly carried out
while the third author was a CNRS postdoc
at the {\it F\'ed\'eration Denis Poisson\/} Orl\'eans--Tours}

\date{\today}



\newtheorem{lemma}{Lemma}[section]
\newtheorem{theorem}[lemma]{Theorem}
\newtheorem{prop}[lemma]{Proposition}
\newtheorem{corollary}[lemma]{Corollary}

\newtheorem{remark}[lemma]{Remark}
\newtheorem{definition}[lemma]{Definition}

\newcommand\1{1\hskip-.95mm\text{I}}
\newcommand\const{\text{const.}}
\renewcommand\Im{\operatorname{Im}}
\renewcommand\Re{\operatorname{Re}}
\newcommand\ssb{\hskip-.25mm}
\newcommand\ssf{\hskip.25mm}

\newcommand{\R}{\mathbb{R}}

\begin{document}
 
\begin{abstract}  
We study the dispersive properties of the wave equation associated
with the shifted Laplace--Beltrami operator on Damek--Ricci spaces,
and deduce Strichartz estimates for a large family of admissible pairs.  
As an application, we obtain global well--posedness results
for the nonlinear wave equation.
\end{abstract}

\maketitle

\section{Introduction}
 
The aim of this paper is to study the dispersive properties of the linear wave equation 
on Damek--Ricci spaces and their application to nonlinear Cauchy problems.  

For the linear wave equation on $\mathbb{R}^{n}$
\begin{equation}\label{WaveEuclidean}
\begin{cases}
&\partial_{\ssf t}^{\ssf 2} u(t,x)-\Delta_{\ssf x}u(t,x)=F(t,x)\,,\\
&u(0,x)=f(x)\,,\\
&\partial_{\ssf t}|_{t=0}\,u(t,x)=g(x)\,,\\
\end{cases}
\end{equation}
the theory is well established\,;
the dispersive $L^1\!\to\!L^{\infty}$ estimates are classical,
while Strichartz estimates were proved by \cite{GV} and \cite{KT}. 
These estimates serve as main tools
to study the corresponding nonlinear problems
and to prove local and global existence 
with either small or large initial data.
In particular, for the semilinear wave equation
\begin{equation}\label{SLWEuclidean}
\begin{cases}
&\partial_{\ssf t}^{\ssf 2} u(t,x)-\Delta_{\ssf x}u(t,x)=F(ut(,x))\,,\\
&u(0,x)=f(x)\,,\\
&\partial_{\ssf t}|_{t=0}\,u(t,x)=g(x)\,,\\
\end{cases}
\end{equation}
with
\begin{equation}\label{eq:nonlinin}
F(u)\sim|u|^{\gamma}\quad\text{near 0}\,,
\end{equation}
a fairly complete theory of well--posedness for small initial data exists.
The results depend on the space dimension $n$\ssf.
After the pioneering work \cite{J} of John in dimension $n\!=\!3$\ssf,
Strauss conjectured in \cite{Stra} that
the problem \eqref{SLWEuclidean} is globally well--posed
for small initial data provided
\begin{equation}\label{StraussExponent}\textstyle
\gamma>\gamma_{0}(n)=\frac12\ssb+\ssb\frac1{n-1}
+\sqrt{\bigl(\frac12\ssb+\ssb\frac1{n-1}\bigr)^2\ssb+\frac2{n-1}\ssf}
\qquad(\ssf n\!\ge\!2\ssf)\ssf.
\end{equation}
The negative part of the conjecture was verified in \cite{Si} by Sideris,
who proved blow up for generic data when $\gamma\!<\!\gamma_{0}(n)$
(and nonlinearities satisfying $F(u)\!\gtrsim\!|u|^{\gamma}$).
The positive part of the conjecture was also verified for any dimension in several steps 
(see e.g.
\cite{KP}
\cite{LS},\cite{GLS},
\cite{DGK},
as well as \cite{G} for a survey
and \cite{DF1}, \cite{DF2} for related results).

Several attemps have been made to extend
Strichartz estimates for dispersive equations
from Euclidean spaces to other settings.
In this paper we consider the \textit{shifted\/} wave equation
\begin{equation}\label{shiftedWave}
\begin{cases}
&\partial_{\ssf t}^{\,2}u(t,x)-(\Delta_{S}\ssb+\ssb Q^2/4)\,u(t,x)=F(t,x)\,\\
&u(0,x)=f(x)\,,\\
&\partial_t|_{t=0}\,u(t,x)=g(x)\,,
\end{cases}
\end{equation}
on \textit{Damek--Ricci spaces\/} $S$
(also known as \textit{harmonic\/} $N\!A$ \textit{groups\/}).
Recall  that these spaces are solvable extensions \ssf$S\!=\!N\!\ltimes\!\R^+$
of Heisenberg type groups $N$\ssb,
equipped with an invariant Riemannian structure\,;
$\Delta_S$ denotes the associated Laplace--Beltrami operator,
whose $L^2$ spectrum is the half line
$\bigl(-\ssf\infty\ssf;-\ssf Q^2\ssb/4\ssf\bigl]$,
and $Q$ the homogeneous dimension of $N$\ssb. 
As Riemannian manifolds, these solvable Lie groups include
all symmetric spaces of the noncompact type and rank one\,;
they are all harmonic but most of them are not symmetric,
thus providing counterexamples to the Lichnerowicz conjecture \cite{DR1}. 
We refer to Section \ref{Notation}
for more details about their structure and analysis thereon.

The Cauchy problem \eqref{shiftedWave}
was considered by Tataru \cite{Ta} and by Ionescu \cite{I1}.
Tataru obtained sharp dispersive $L^{q'}\hspace{-1mm}\to\!L^q$ estimates
for the operators
$$\textstyle
\cos\bigl(\ssf t\ssf\sqrt{\ssb-\Delta_{S}\!-\ssb Q^2/4\ssf}\ssf\bigr)
\quad\text{and}\quad
\frac{\sin\ssf(\ssf t\ssf\sqrt{\ssb-\Delta_{S}-\ssf Q^2/4\ssf}\ssf)}
{\sqrt{\ssb-\Delta_{S}-\ssf Q^2/4}}
$$
when $S$ is a real hyperbolic space,
while Ionescu investigated $L^q\!\to\!L^q$ estimates 
for these operators
when $S$ is a rank one symmetric space.

In \cite{APV2} we derived Strichartz estimates
for the Cauchy problem \eqref{shiftedWave} 
when $S$ is a real hyperbolic space.
Our aim here is to extend the results obtained in \cite{APV2} 
to the larger class of Damek--Ricci spaces.
The difficulty is due to the fact that
Damek--Ricci spaces  are nonsymmetric in general,
so that some of the proofs given in \cite{APV2}
do not work in this context. 
Despite this difficulty,
we are able to obtain Strichartz estimates
for solutions to the Cauchy problem (\ref{shiftedWave}). 
Corresponding results for the the Schr\"odinger
equation were obtained \cite{AP} and \cite{APV1} (see also \cite{P}).

In Section \ref{GWP} we apply our Strichartz estimates 
to obtain global well--posedness results
for the nonlinear wave equation with small initial data and low regularity.
Notice that this result is new even for hyperbolic spaces,
since in \cite{APV2} we only discussed local well--posedness. 
An interesting new feature,
which differentiates our results from the Euclidean case,
is the absence of a lower critical exponent
for power--like nonlinearities on Damek--Ricci spaces.
Indeed, for $\gamma\!>\!1$ arbitrarily close to $1$,
we are able to prove global existence for the problem
\begin{equation}\label{shiftedWaveSL}
\begin{cases}
&\partial_{\ssf t}^{\,2}\ssf u(t,x)-(\Delta_{S}\!+\ssb Q^2/4)\,u(t,x)=F(u(t,x))\,,\\
&u(0,x) = f(x)\,,\\
&\partial_t|_{t=0}\,u(t,x)=g(x)\,,
\end{cases}
\end{equation}
with small initial data and nonlinearities $F$ satisfying
\begin{equation*}
|F(u)|\le C\,|u|^\gamma
\quad\text{and}\quad
|\ssf F(u)\ssb-\ssb F(v)\ssf|\ssf\le\ssf
C\,(\ssf|u|^{\gamma-1}\!+\ssb|v|^{\gamma-1}\ssf)\,|\ssf u\ssb-\ssb v\ssf|\ssf.
\end{equation*}
Recall that Tataru \cite{Ta} proved global existence on hyperbolic spaces
for small smooth initial data,
provided the power $\gamma$ is greater
than the Strauss critical exponent \eqref{StraussExponent}.
Thus, by combining our results with \cite{Ta},
we see that the Cauchy problem \eqref{shiftedWaveSL} is well posed
for small smooth initial data and any power $\gamma\!>\!1$.
Notice moreover that in Theorem \eqref{WPL2}
we allow for small initial data with low regularity,
arbitrarily close to the critical one in the Euclidean case, 
which is determined by concentration and scaling arguments.

\section{Damek--Ricci spaces}
\label{Notation}

In this section we recall the definition of $H$--type groups,
describe their Damek--Ricci extensions,
and recall the main results of spherical analysis on these spaces.
For the details we refer the reader to
\cite{ADY, CDKR1, CDKR2, DR1, DR2, R}.
\smallskip

Let $\mathfrak{n}$ be a Lie algebra equipped  with an inner product $\langle\cdot,\cdot\rangle$ and denote by $|\cdot |$ the corre\-spon\-ding norm. Let $\mathfrak{v}$ and $\mathfrak{z}$ be complementary orthogonal subspaces of $\mathfrak{n}$ such that $[\mathfrak{n},\mathfrak{z} ]=\{0\}$ and $[\mathfrak{n},\mathfrak{n}]\subseteq \mathfrak{z}$.
According to Kaplan \cite{Kapl},
the algebra $\mathfrak{n}$ is of $H$--type if,
for every $Z$ in $\mathfrak{z}$ of unit length,
the map $J_Z:\mathfrak{v}\to \mathfrak{v}$,
defined by
$$
\langle J_ZX,Y\rangle\,=\,\langle Z,[X,Y]\rangle
\qquad\forall X, Y\in \mathfrak{v}\,,
$$
is orthogonal.
The connected and simply connected Lie group $N$
associated to $\mathfrak{n}$
is called an $H$--type group.
We identify $N$ with its Lie algebra $\mathfrak{n}$
via the exponential map
\begin{eqnarray*}
\mathfrak{v}\times\mathfrak{z}&\longrightarrow&N\\
(X,Z)&\longmapsto&\exp(X+Z)\,.
\end{eqnarray*}
Thus multiplication in $N$ reads
$$\textstyle
(X,Z)\ssf(X',Z')=\bigl(X\ssb+X',Z\ssb+\ssb Z'\!+\ssb\frac12\,[X,X']\bigr)
\qquad\forall X,\,X'\in\mathfrak{v}\quad\forall Z,\,Z'\in\mathfrak{z}\,.
$$
The group $N$ is a two-step nilpotent group
with Haar measure $dX\ssf dZ$\ssf.
The number \ssf$Q\ssb=\ssb\frac m2\ssb+\ssb k$\ssf,
where \ssf$m$ \ssf and \ssf$k$ \ssf denote
the dimensions of \ssf$\mathfrak{v}$ \ssf
and \ssf$\mathfrak{z}$ \ssf respectively,
is called the homogeneous dimension of \ssf$N$.

Let \ssf$S$ \ssf be the semidirect product
\ssf$S\ssb=\ssb N\ssb\ltimes\ssb\mathbb{R}^+$,
defined by
$$\textstyle
(X,Z,a)\ssf(X',Z',a')
=\bigl(X\ssb+a^{\frac12}X',
Z\ssb+a\ssf Z'\!+\ssb\frac12\,a^{\frac12}\ssf[X,X'],
a\ssf a'\bigr)
$$
for all \ssf$(X,Z,a),(X',Z',a')\!\in\!S$\ssf.
We shall denote by \ssf$n$ \ssf
the dimension \ssf$m\ssb+\ssb k\ssb+\!1$ of \ssf$S$\ssf.
Notice that \ssf$m$ \ssf is an even number $\ge\ssb2$
and we shall always assume that $k\!\ge\!1$
(the case when \ssf$k\ssb=\ssb0$
\ssf corresponds to real hyperbolic spaces
and has been investigated in \cite{APV2}). 
This implies that the dimension of the space \ssf$S$
\ssf is \ssf$n\!\ge\!4$\ssf. 

The group $S$ is nonunimodular.
Indeed the right and left Haar measures on $S$ are given respectively by 
$$
d\rho(X,Z,a)=a^{-1}\,dX\,d Z\,da
\qquad{\textrm{and}}\qquad
d\mu(X,Z,a)=a^{-(Q+1)}\,d X\,d Z\,d a \,.$$  
Then the modular function is \ssf$\delta(X,Z,a)\ssb=\ssb a^{-Q}$.

We equip $S$ with the left invariant Riemannian metric induced by the inner product
$$
\langle (X,Z,\ell),(X',Z',\ell')\rangle
=\langle X,X'\rangle+\langle Z,Z'\rangle+\ell\,\ell'\,,
$$
on the Lie algebra $\mathfrak{s}$ of $S$.
For every $x\!\in\!S$,
we shall denote by $r(x)$
the distance between the point $x$ and the identity $e$ of $S$
and by $a(x)$ the $A$--component of $x$,
i.e. the element \ssf$a(x)\!\in\!\mathbb{R}^+$
such that $x\ssb=\ssb(X,Z,a(x))$,
with $X\!\in\!\mathfrak{v}$, $Z\!\in\!\mathfrak{z}$\ssf.
The following useful inequality holds (see \cite[formula (1.20)]{ADY})\,:
\begin{equation}\label{loga}
|\log a(x)|\leq r(x)\qquad\forall\,x\!\in\!S\,.
\end{equation} 
The Riemannian measure is the left Haar measure $\mu$ introduced above 
and we denote by $\Delta_S$ the Laplace--Beltrami operator
associated with this Riemannian structure on $S$. 

A radial function on $S$ is a function
that depends only on the distance from the identity.
If $f$ is radial, then by \cite[formula (1.16)]{ADY}
$$
\int_S\,d\mu\,f=\int_0^{\infty}dr\,f(r)\,V(r)\,,
$$
where 
\begin{equation}\label{density}\textstyle
V(r)=\ssf2^{m+k}\,\sinh^{m+k}\frac r2\,\cosh^{k}\frac r2
\qquad\forall\,r\!\in\!\mathbb{R}^+.
\end{equation}
Let $\pi$ denote the radialisation operator defined in \cite[page 150]{A1} which associates to each function $f$ in $C^{\infty}(S)$ a radial function on $S$. More precisely, 
$$
\pi f(r)=\const\,\int_{\partial B(\mathfrak{s})}d\sigma\,f(r\sigma)
\qquad\forall\,r\!\in\!\mathbb{R}^+,
$$
where $\partial B(\mathfrak{s})$ is the unit sphere in $\mathfrak s$
and $d\sigma$ denotes the surface measure on it.  

The spherical functions $\varphi_{\lambda}$ on $S$
are normalized eigenfunctions of $\Delta_S$\,:
\begin{equation*}\begin{cases}
\;\Delta_{S}\,\varphi_\lambda
=-\bigl(\lambda^2\!+\ssb\frac{Q^2}4\bigr)\,\varphi_\lambda\,,\\
\;\varphi_\lambda(e)=1\,,
\end{cases}\end{equation*}
where \ssf$\lambda\!\in\!\mathbb{C}$ (see \cite[formula (2.6)]{ADY}). 
In the sequel we shall use various properties of the spherical functions,
which we now summarize.
We refer to \cite{ADY, DR2} for more details. 

All spherical functions are of the form 
\begin{equation}\label{radialisationformula}
\varphi_{\lambda}=\pi(\delta^{i\lambda/Q-1/2})=\pi(a(\cdot)^{-i\lambda+Q/2})
\qquad\forall\,\lambda\!\in\!\mathbb{C}\,,
\end{equation}
where $\delta$ is the modular function. This easily implies that 
\begin{equation}\label{philphi0}
|\varphi_{\lambda}(r)|\lesssim\varphi_0(r)
\qquad\forall\,\lambda\!\in\!\mathbb{C}\ssf,\,\forall\,r\!\in\!\mathbb{R}^+. 
\end{equation}
Moreover, it is well known that
\begin{equation}\label{phi0}
\varphi_0(r)\lesssim(1\!+\ssb r)\,e^{-\frac{Q}2r}
\qquad\forall\,r\!\in\!\mathbb{R}^+.
\end{equation}

The asymptotic behavior of the spherical functions is given by
\begin{equation*}\textstyle
\varphi_\lambda(r)
=\mathbf{c}(\lambda)\,\Phi_\lambda(r)
+\mathbf{c}(-\lambda)\,\Phi_{-\lambda}(r)
\qquad\forall\,\lambda\!\in\!\mathbb{C}\!\smallsetminus\!\frac i2\ssf\mathbb{Z}\,,
\end{equation*}
where
\begin{equation}\label{cfunction}\textstyle
\mathbf{c}(\lambda)= 
\Gamma\bigl(\frac n2\bigr)\,
2^{\ssf Q-2i\lambda}\,
\frac{\Gamma(2\,i\ssf\lambda)\vphantom{\big|}}
{\Gamma(i\ssf\lambda\ssf+\ssf\frac Q2)\,
\Gamma(i\ssf\lambda\ssf+\ssf\frac m4\ssf+\ssf\frac 12)
\vphantom{\big|}}
\end{equation}
and
\begin{equation*}\textstyle
\Phi_{\lambda}(r)=(2\cosh\frac r2)^{\ssf i\ssf2\lambda-Q}\,
{}_2F_1\bigl(\frac Q2\!-\!i\lambda\ssf,\frac m4\!-\!\frac12\!-\!i\lambda\ssf;
1\!-\!2\ssf i\lambda\ssf;(\cosh\frac r2)^{-2}\bigr)
\end{equation*}
(see \cite[pp. 7--8]{Ko}).
On one hand,
$\Phi_\lambda$ is another radial eigenfunction of \ssf$\Delta_S$
for the same eigenvalue \ssf$-\bigl(\lambda^2\!+\ssb\frac{Q^2}4\bigr)$,
i.e.
\begin{equation}\begin{aligned}\label{EigFunctEq}
0&\textstyle
=\bigl\{\ssf\Delta_S\ssb+\ssb\frac{Q^2}4\!+\ssb\lambda^2\bigr\}\,\Phi_\lambda(r)
=\bigl\{\ssf\partial_r^{\ssf2}\!+\ssb\frac{V'(r)}{V(r))}\,\partial_r\ssb
+\ssb\frac{Q^2}4\ssb+\ssb\lambda^2\bigr\}\,\Phi_\lambda(r)\\
&\textstyle
=V(r)^{-\frac12}\,
\bigl\{\ssf\partial_r^{\ssf2}\!-\ssb\omega(r)\ssb+\ssb\lambda^2\bigr\}\,
\bigl\{\ssf V(r)^{\frac12}\,\Phi_\lambda(r)\bigr\}\,,
\end{aligned}\end{equation}
where
\begin{equation}\begin{aligned}\label{omega}
\omega(r)
&\textstyle
=V(r)^{-\frac12}\,\partial_r^{\ssf2}\,V(r)^{\frac12}-\frac{Q^2}4\\
&\textstyle
=\frac14\ssf\frac m2\ssf\bigl(Q\!-\!1\bigr)\ssf
\bigl(\sinh\frac r2\bigr)^{-2}\!
+\ssb\frac k2\ssf\bigl(\frac k2\!-\!1\bigr)\ssf
\bigl(\sinh r\bigr)^{-2}\\
&=\,\sum\nolimits_{\ssf j=1}^{+\infty}\,\omega_j\,e^{-jr}
\qquad\text{with \;}\omega_j\!=\ssb\text{O}(j)\ssf.
\end{aligned}\end{equation}
On the other hand, the function $\Phi_\lambda$ can be expanded as follows\,:
\begin{equation}\begin{aligned}\label{expPhilambda}
\Phi_\lambda(r)\,&\textstyle=\,
{\displaystyle\sum\nolimits_{\ssf\ell=0}^{+\infty}}\,
\frac{\Gamma(Q/2-i\lambda+\ell)}{\Gamma(Q/2-i\lambda)}\,
\frac{\Gamma(m/4+1/2-i\lambda+\ell)}{\Gamma(m/4+1/2-i\lambda)}\,
\frac{\Gamma(1-2\ssf i\lambda+\ell)}{\Gamma(1-2\ssf i\lambda)}\,
\frac{2^{\ssf2\ssf\ell}}{\ell\,!}\,
(2\cosh\frac r2)^{2\ssf i\lambda-Q-2\ssf\ell}\\
\textstyle
&=\,2^{-\frac k2}\,V(r)^{-\frac12}\,
{\displaystyle\sum\nolimits_{\ssf \ell=0}^{+\infty}}\,
\Gamma_\ell(\lambda)\,e^{\ssf(i\lambda-\ell)\ssf r}
\qquad\text{as \;}r\ssb\to\ssb+\infty\ssf.
\end{aligned}\end{equation}
By combining \eqref{EigFunctEq}, \eqref{omega}, \eqref{expPhilambda},
the coefficients $\Gamma_\ell$ are shown
to satisfy the recurrence formula
\begin{equation}\label{recurrence}\begin{cases}
\;\Gamma_0\ssb=\ssb1\ssf,\\
\;\ell\,(\ssf\ell\ssb-\ssb i\,2\ssf\lambda)\,\Gamma_\ell(\lambda)
={\displaystyle\sum\nolimits_{\ssf j=0}^{\ssf\ell-1}}\;
\omega_{\ssf\ell-j}\,\Gamma_j(\lambda)
\quad\forall\;\ell\!\in\!\mathbb{N}^*.
\end{cases}\end{equation} 
It is well known (see e.g. \cite[Theorem 3.2]{A1}) that
there exist nonnegative constants \ssf$C$ and \ssf$d$ \ssf
such that
\begin{equation}\label{Gamma}
|\ssf\Gamma_{\ell}(\lambda)\ssf|\le C\,(1\!+\ssb\ell\ssf)^d\,,
\end{equation}
for all \ssf$\ell\!\in\!\mathbb{N}$ \ssf
and for all \ssf$\lambda\!\in\!\mathbb{C}$
\ssf with \ssf $\Im\lambda\ssb\ge\!-\ssf|\ssb\Re\lambda\ssf|$\ssf.
We shall need the following improved estimates.

\begin{lemma}\label{l: derGammaelle}
Let \,$0\!<\!\varepsilon\!<\!1$
and \,$\Omega_\varepsilon\ssb=\ssb\{\,\lambda\!\in\ssb\mathbb{C}
\mid|\ssb\Re\lambda\ssf|\!\le\ssb\varepsilon\,|\lambda|\ssf,\,
\Im\lambda\!\le\!-\ssf\frac{1\ssf-\ssf\varepsilon}2\,\}$\ssf.
Then, there exists a positive constant \,$d$ and,
for every \,$h\!\in\!\mathbb{N}$\ssf,
a positive constant \,$C$ such that
\begin{equation}\label{derestimate}\textstyle
|\,\partial_\lambda^{\,h}\ssf\Gamma_\ell(\lambda)\ssf|
\le C\,\ell^{\ssf d}\,(1\!+\!|\lambda|\ssf)^{-h-1}
\quad\forall\;\ell\!\in\!\mathbb{N}^*,
\,\forall\,\lambda\!\in\!\mathbb{C}\!\smallsetminus\!\Omega_\varepsilon\ssf.
\end{equation} 
\end{lemma}

\begin{proof}
The case \ssf$h\!\in\!\mathbb{N}^*$
follows by Cauchy's formula
from the case \ssf$h\!=\!0$\ssf,
that we prove now.
On one hand,
there exists \ssf$A\!\ge\!0$
\ssf such that
\begin{equation*}
|\,\omega_j\ssf|\le A\;j
\qquad\forall\;j\!\in\!\mathbb{N}^*.
\end{equation*}
On the other hand,
there exists \ssf$B\!>\!0$ \,such that
\begin{equation*}
|\,\ell\ssb-\ssb i\,2\ssf\lambda\ssf|
\ge\ssb B\,\max\,\{\ssf\ell,1\!+\!|\lambda|\ssf\}
\quad\forall\;\ell\!\in\!\mathbb{N}^*,
\,\forall\,\lambda\!\in\!\mathbb{C}\!\smallsetminus\!\Omega_\varepsilon\ssf.
\end{equation*}
Choose \ssf$C\!=\!2\ssf A\ssf/B$ and \ssf$d\ssb\ge\!1$
\ssf such that \ssf$d\!+\!1\ssb\ge\ssb C$\ssf.
For \ssf$\ell\ssb=\!1$\ssf,
we have \ssf$\Gamma_1(\lambda)\ssb
=\ssb\frac{\omega_1}{1\ssf-\,i\,2\ssf\lambda}$\ssf,
which implies
\begin{equation*}\textstyle
|\,\Gamma_1(\lambda)\ssf|
\le\frac AB\ssf\frac1{1\ssf+\ssf|\lambda|}
\le\frac C{1\ssf+\ssf|\lambda|}\,,
\end{equation*}
as required. For \ssf$\ell\!>\!1$\ssf, we have
\begin{equation*}\textstyle
\Gamma_\ell(\lambda)
=\frac{\omega_\ell}{\ell\,(\ssf\ell\ssf-\ssf i\,2\ssf\lambda)}
+\frac1{\ell\,(\ssf\ell\ssf-\ssf i\,2\ssf\lambda)}\,
{\displaystyle\sum\nolimits_{\ssf0<j<\ell}}\,
\omega_{\ssf\ell-j}\,\Gamma_j(\lambda)\,,
\end{equation*}
which implies
\begin{equation*}\begin{aligned}
|\,\Gamma_\ell(\lambda)\ssf|
&\textstyle
\le\frac AB\ssf\frac\ell{1\ssf+\ssf|\lambda|}
+\frac AB\ssf\frac1{\ell^2}\,
{\displaystyle\sum\nolimits_{\ssf0<j<\ell}}\,
(\ell\!-\!j)\,\frac{C\,j^{\ssf d}}{1\ssf+\ssf|\lambda|}\\
&\textstyle
\le\frac C2\ssf\frac{\ell^{\ssf d}}{1\ssf+\ssf|\lambda|}
+\frac C2\ssf\frac{\ell^{\ssf d}}{1\ssf+\ssf|\lambda|}\,\frac C\ell\,
{\displaystyle\sum\nolimits_{\ssf0<j<\ell}}\,
\bigl(\frac j\ell\bigr)^d\\
&\textstyle
\le\,C\,\frac{\ell^{\ssf d}}{1\ssf+\ssf|\lambda|}\,.
\end{aligned}\end{equation*}
\end{proof}
\medskip

The spherical Fourier transform \ssf$\mathcal{H}f$
of an integrable radial function \ssf$f$ on \ssf$S$ \ssf is defined by  
$$
\mathcal{H}f(\lambda)=\int_Sd\mu\,f\,\varphi_{\lambda}\,.
$$
For suitable radial functions \ssf$f$ on \ssf$S$,
an inversion formula and a Plancherel formula hold:  
$$f(x)=c_S\int_{0}^{\infty}d\lambda\,| {\mathbf{c}} (\lambda) |^{-2} \mathcal{H} f (\lambda)\,\varphi _{\lambda}(x)\, \, \qquad \forall x\in S\,,$$
and
$$\int_Sd\mu \,|f|^2\,  =c_S\int_0^{\infty}d\lambda\,|{\mathbf{c}}(\lambda)|^{-2}\,|\mathcal{H} f (\lambda)|^2\, \,,$$
where the constant $c_S$ depends only on $m$ and $k$. It is well known that
\begin{equation}\label{HCest}
| {\mathbf{c}}(\lambda)|^{-2}\lesssim |\lambda|^{-2}\,(1+|\lambda|)^{n-3}\qquad\forall \lambda\in\mathbb{R}\,.
\end{equation}

In the sequel we shall use the fact that
\ssf$\mathcal{H}\ssb=\ssb\mathcal{F}\ssb\circ\ssb\mathcal{A}$\ssf,
where $\mathcal{A}$ denotes the Abel transform
and $\mathcal{F}$ denotes the Fourier transform on the real line.
Actually we shall use the factorization
\ssf$\mathcal{H}^{-1}\!=\ssb\mathcal{A}^{-1}\!\circ\ssb\mathcal{F}^{-1}$.
For later use, let us recall the inversion formulae
for the Abel transform \cite[formula (2.24)]{ADY},
which involve the differential operators
\begin{equation*}\textstyle
\mathcal{D}_1\ssb=-\ssf\frac1{\sinh r}\,\frac{\partial}{\partial r}
\quad\text{and}\quad
\mathcal{D}_2\ssb=-\ssf\frac1{\sinh(r/2)}\,\frac{\partial}{\partial r}\,.
\end{equation*}
If $k$ is even, then
\begin{equation}\label{inv1}
\mathcal{A}^{-1}\ssb f(r)
=a_S^e\,\mathcal{D}_1^{k/2}\,\mathcal{D}_2^{m/2}f\ssf(r)\,,
\end{equation}
where \ssf$a_S^e=2^{-(3m+k)/2}\,\pi^{-(m+k)/2}$,
while, if $k$ is odd, then
\begin{equation}\label{inv2}
\mathcal{A}^{-1}\ssb f(r)=a_S^o\int_r^{\infty}
\mathcal{D}_1^{(k+1)/2}\,\mathcal{D}_2^{m/2}f\ssf(s)\,{\mathrm{d}}\nu(s)\,,
\end{equation}
where \ssf$a_S^o= 2^{-(3m+k)/2}\,\pi^{-n/2}$ \ssf
and \ssf${\mathrm{d}}\nu(s)
=(\cosh s\ssb-\ssb\cosh r)^{-1/2}\sinh s\,{\mathrm{d}}s$.

\section{Sobolev spaces and conservation of energy}
\label{Sobolev}

Let us first introduce inhomogeneous Sobolev spaces on a Damek--Ricci space,
which will be involved in the conservation laws, in the dispersive estimates
and in the Strichartz estimates for the shifted wave equation.
We refer to \cite{Tr2} for more details about function spaces on Riemannian manifolds.

Let \,$1\!<\!q\!<\!\infty$
\,and \,$\sigma\!\in\!\R$\ssf.
By definition,
$H_q^\sigma(S)$ is the image of $L^q(S)$
under $(-\Delta_{S})^{-\frac\sigma2}$
(in the space of distributions on $S$),
equipped with the norm
\begin{equation*}
\|\ssf f\ssf\|_{H_q^\sigma}
=\ssf\|\ssf(-\Delta_{S})^{\frac\sigma2}f\ssf\|_{L^q}\,.
\end{equation*}
In this definition, we may replace $-\Delta_{S}$ by
$-\Delta_{S}\!-\!\frac{Q^2}4\hspace{-1mm}+\!\frac{\widetilde{Q}^2}4$,
where \ssf$\widetilde{Q}\!>\!Q$ and we set
\begin{equation*}\textstyle
\widetilde D=\bigl(\ssb-\ssf\Delta_{S}\!-\!\frac{Q^2}4\hspace{-1mm}
+\!\frac{\widetilde{Q}^2}4\bigr)^{\frac12}\,.
\end{equation*}
Thus $H_q^\sigma(S)\!=\!\widetilde D^{-\sigma}L^q(S)$
and $\|\ssf f\ssf\|_{H_q^\sigma}\!
\sim \ssb\|\ssf\widetilde D^{\ssf\sigma\ssb}f\ssf\|_{L^q}$.
If \ssf$\sigma\ssb=\ssb N$ \ssf is a nonnegative integer,
then $H_q^\sigma(S)$ co\"{\i}ncides with the Sobolev space
\begin{equation*}
W^{N,q}(S)=\{\,f\!\in\!L^q(S)\mid\nabla^j\ssb f\!\in\! L^q(S)
\hspace{2mm}\forall\;1\!\le\!j\!\le\!N\,\}
\end{equation*}
defined in terms of covariant derivatives and equipped with the norm
\begin{equation*}
\|f\|_{W^{N,q}}=\ssf\sum\nolimits_{\ssf j=0}^{\,N}
\|\ssf\nabla^j\ssb f\ssf\|_{L^q}\,.
\end{equation*}
By following the same proof of \cite[Proposition 3.1]{APV2}
we obtain the following Sobolev embedding Theorem. 

\begin{prop}\label{SET}
Let \,$1\!<\!q_1\!<\!q_2\!<\!\infty$ \ssf
and \,$\sigma_1,\sigma_2\!\in\!\R$ such that
\,$\sigma_1\!-\!\frac n{q_1}\ssb\ge\ssb\sigma_2\!-\!\frac n{q_2}$.
Then
\begin{equation*}
H_{q_1}^{\sigma_1}(S)\subset H_{q_2}^{\sigma_2}(S)\,.
\end{equation*}
By this inclusion, we mean that
there exists a constant \,$C\!>\!0$ \ssf
such that
\begin{equation*}
\|f\|_{H_{q_2}^{\sigma_2}}\le C\,\|f\|_{H_{q_1}^{\sigma_1}}
\qquad\forall\;f\!\in\!C_c^\infty(S)\,.
\end{equation*}
\end{prop}
\medskip

Beside the $L^q$ Sobolev spaces $H_{\ssf q}^\sigma(S)$,
our analysis of the shifted wave equation on $S$
involves the following $L^2$ Sobolev spaces\,:
\begin{equation*}
H^{\sigma,\tau}(S)=\ssf\widetilde D^{-\sigma}D^{-\tau}L^2(S)\ssf,
\end{equation*}
where \ssf$D\!
=\!\bigl(\ssb-\ssf\Delta_{S}\!-\!\frac{Q^2}4\bigr)^{\frac12}$\ssf,
\ssf$\sigma\!\in\!\R$ \ssf and \ssf$\tau\!<\!\frac32$
\ssf(actually we are only interested in the cases
\ssf$\tau\ssb=\ssb0$ \ssf and \ssf$\tau\ssb=\ssb\pm\ssf\frac12$\ssf).
Notice that
\begin{equation*}\begin{cases}
\,H^{\sigma,\tau}(S)\ssb
=\ssb H_{\ssf2}^{\sigma}(S)
&\text{if \,}\tau\ssb=\ssb0\ssf,\\
\,H^{\sigma,\tau}(S)\ssb
\subset\ssb H_{\ssf2}^{\sigma+\tau}(S)
&\text{if \,}\tau\ssb<\ssb0\ssf,\\
\,H^{\sigma,\tau}(S)\ssb
\supset\ssb H_{\ssf2}^{\sigma+\tau}(S)
&\text{if \,}0\ssb<\ssb\tau\ssb<\ssb\frac32\ssf.\\
\end{cases}\end{equation*}

\begin{lemma}
If \;$0\ssb<\!\tau\!<\frac32$\ssf, then
\begin{equation*}
H^{\sigma,\tau}(S)\subset
H_2^{\sigma+\tau}(S)+H_{2^+}^\infty(S)\ssf,
\end{equation*}
where $H_{2^+}^\infty(S)\ssb=
\bigcap_{\,\substack{s\in\R\\q>2}\vphantom{\big|}}\ssb
H_q^s(S)$
{\rm (}recall that $H_q^s(S)$ is decreasing
as $q\!\searrow\!2$ and $s\!\nearrow\!+\infty)$.
\end{lemma}

\begin{proof}
See \cite[Lemma 3.2]{APV2}. 
\end{proof}

Let us next introduce the energy
\begin{equation}\label{energy}\textstyle
E(t)=\frac12{\displaystyle\int_{S}}\hspace{-1mm}\,d\mu(x)\,
\bigl\{\ssf|\partial_{\ssf t}u(t,x)|^2\ssb+|D_xu(t,x)|^2\ssf\bigr\} 
\end{equation}
for solutions to the homogeneous Cauchy problem
\begin{equation}\label{wavehom}
\begin{cases}
&\partial_{\ssf t}^{\ssf2}u-\bigl(\Delta_{S}\!+\!\frac{Q^2}4\bigr)\ssf u=0\,\\
&u(0,x)=f(x)\,\\
&\partial_{\ssf t}|_{t=0}\,u(t,x)=g(x)\,.
\end{cases}
\end{equation}
It is easily verified that $\partial_{\ssf t}E(t)\!=\!0$\ssf,
hence \eqref{energy} is conserved.
In other words,
for every time \ssf$t$ \ssf in the interval of definition of \ssf$u$\ssf,
\begin{equation*}
\|\ssf\partial_{\ssf t}u(t,x)\|_{L_x^2}^{\ssf2}\ssb
+\|D_xu(t,x)\|_{L_x^2}^{\ssf2}
=\|g\|_{L^2}^{\ssf2}\ssb+\|D\ssb f\|_{L^2}^{\ssf2}\,.
\end{equation*}
Let \ssf$\sigma\!\in\!\R$ \ssf and \ssf$\tau\!<\!\frac32$\ssf.
By applying the operator \ssf$\tilde{D}^\sigma\ssb D^\tau$
to  (\ref{wavehom}), we deduce that
\begin{equation*}
\|\,\partial_{\ssf t}\ssf\tilde{D}_x^\sigma D_x^\tau\ssf
u(t,\cdot)\ssf\|_{L_x^2}^{\ssf2}
+\|\ssf\tilde{D}_x^\sigma D_x^{\tau+1}\ssf u(t,\cdot)\ssf\|_{L_x^2}^{\ssf2}
=\|\ssf \tilde{D}^\sigma\ssb D^\tau\ssb g\ssf\|_{L^2}^{\ssf2}
+\|\ssf \tilde{D}^\sigma\ssb D^{\tau+1}\ssb f\ssf\|_{L^2}^{\ssf2}\,,
\end{equation*}
which can be rewritten in terms of Sobolev norms as follows\,:
\begin{equation}\label{conservationenergy}
\|\ssf\partial_{\ssf t}u(t,\cdot)\|_{H^{\sigma,\tau}}^{\ssf2}\ssb
+\|u(t,\cdot)\|_{H^{\sigma,\tau+1}}^{\ssf2}\ssb
=\|g\|_{H^{\sigma,\tau}}^{\ssf2}\ssb
+\|f \|_{H^{\sigma,\tau+1}}^{\ssf2}\,.
\end{equation}

\section{Kernel estimates}
\label{Kernel}

In this section we derive pointwise estimates
for the radial convolution kernel \ssf$w_{\,t}^{(\sigma,\tau)}$
of the operator \ssf$W_t^{(\sigma,\tau)}\!
=\ssb D^{-\tau}\tilde{D}^{\ssf\tau-\sigma}e^{\,i\,t\ssf D}$,
for suitable exponents  $\sigma\!\in\!\R$ and $\tau\!\in\![0,\frac{3}2)$\ssf. 
To do so, we follow the strategy used in \cite{APV2} for hyperbolic spaces. 
The difficulty here is that Damek--Ricci spaces are nonsymmetric in general,
so that some of the proofs given in \cite{APV2} do not work in this context. 

By the inversion formula of the spherical Fourier transform,
\begin{equation*}
w_{\,t}^{(\sigma, \tau)}(r)=\const
\int_{\,0}^{+\infty}\hspace{-1mm}d\lambda\;
|\mathbf{c}\hspace{.1mm}(\lambda)|^{-2}\,\lambda^{-\tau}\,
\bigl(\lambda^2\!+\ssb{\textstyle\frac{\widetilde{Q}^2}4}\bigr)^{\!\frac{\tau-\sigma}2}\,
\varphi_\lambda(r)\,e^{\ssf i\ssf t\ssf\lambda}\,.
\end{equation*}
Let us split up
\begin{equation*}\begin{aligned}
w_{\,t}^{(\sigma,\tau)}(r)
&=w_{\,t,0}^{(\sigma,\tau)}(r)+w_{\,t,\infty}^{(\sigma,\tau)}(r)\\
&=\const\int_{\,0}^{\ssf2}\!d\lambda\,
\chi_0(\lambda)\,|\mathbf{c}\hspace{.1mm}(\lambda)|^{-2}\,\lambda^{-\tau}\,
\bigl(\lambda^2\!+\ssb{\textstyle\frac{\widetilde{Q}^2}4}\bigr)^{\!\frac{\tau-\sigma}2}\,
\varphi_\lambda(r)\,e^{\ssf i\ssf t\ssf\lambda}\\
&+\const\int_{\,1}^{+\infty}\hspace{-1mm}d\lambda\,
\chi_\infty(\lambda)\,|\mathbf{c}\hspace{.1mm}(\lambda)|^{-2}
\,\lambda^{-\tau}\,\bigl(\lambda^2\!+\ssb{\textstyle\frac{\widetilde{Q}^2}4}\bigr)^{\!\frac{\tau-\sigma}2}\,
\varphi_\lambda(r)\,e^{\ssf i\ssf t\ssf\lambda}\,,
\end{aligned}\end{equation*}
using smooth cut--off functions $\chi_0$ and $\chi_\infty$ on $[0,+\infty)$
such that $1\!=\ssb\chi_0\ssb+\ssb\chi_\infty$\ssf,
$\chi_0\ssb=\!1$ on $[\ssf0,1\ssf]$ and
$\chi_\infty\!=\!1$ on $[\ssf2,+\infty)$\ssf.
We shall first estimate \ssf$w_{\,t,0}^{(\sigma,\tau)}$
and next a variant of \ssf$w_{\,t,\infty}^{(\sigma,\tau)}$.
The kernel \ssf$w_{\,t,\infty}^{(\sigma,\tau)}$ has indeed
a logarithmic singularity on the sphere \ssf$r\!=\ssb t$
\ssf when \ssf$\sigma\ssb=\ssb\frac{n+1}2$.
We bypass this problem
by considering the analytic family of operators
\begin{equation*}\textstyle
\widetilde{W}_{\,t,\infty}^{\ssf(\sigma,\tau)}
=\frac{e^{\ssf\sigma^2}}{\Gamma(\frac{n+1}2-\sigma)}\;
\chi_\infty(D)\,D^{-\tau}\,\tilde{D}^{\ssf\tau-\sigma}\,e^{\,i\,t\ssf D}
\end{equation*}
in the vertical strip \ssf$0\!\le\!\Re\sigma\!\le\!\frac{n+1}2$
\ssf and the corresponding kernels
\begin{equation}\label{ISFT}\textstyle
\widetilde{w}_{\,t,\infty}^{\ssf(\sigma,\tau)}(r)
=\frac{e^{\ssf\sigma^2}}{\Gamma(\frac{n+1}2-\sigma)}\,
{\displaystyle\int_{\,1}^{+\infty}}\hskip-1mm
d\lambda\,\chi_\infty(\lambda)\,
|\mathbf{c}\hspace{.1mm}(\lambda)|^{-2}\,\lambda^{-\tau}\,
\bigl(\lambda^2\!+\ssb{\textstyle\frac{\widetilde{Q}^2}4}\bigr)^{\!\frac{\tau-\sigma}2}\,
e^{\,i\ssf t\ssf\lambda}\,\varphi_\lambda(r)\,.
\end{equation}
Notice that the Gamma function,
which occurs naturally in the theory of Riesz distributions,
will allow us to deal with the boundary point \ssf$\sigma\!=\!\frac{n+1}2$,
while the exponential function
yields boundedness at infinity in the vertical strip.
Notice also that, once multiplied by \ssf$\chi_\infty(D)$\ssf,
the operator \ssf$D^{-\tau}\tilde{D}^{\ssf\tau-\sigma}$
\ssf behaves like \ssf$\tilde{D}^{-\sigma}$\ssf.

 \subsection{Estimate of
\,$w_{\ssf t}^{\ssf0}\ssb=\ssb w_{\,t,0}^{(\sigma,\tau)}$.}
\label{KernelEstimatewt0}

\begin{theorem}\label{Estimatewt0}
Let \,$\sigma\!\in\!\R$ and \,$\tau\!<\!2$\ssf.
The following pointwise estimates hold for the kernel
\,$w_{\ssf t}^{\ssf0}\ssb=\ssb w_{\,t,0}^{(\sigma,\tau)}:$
\begin{itemize}
\item[(i)]
Assume that \,$|t|\!\le\ssb2$\ssf.
Then, for every \,$r\ssb\ge\ssb0$\ssf,
\begin{equation*}
|\ssf w_{\ssf t}^{\ssf0}(r)|\ssf\lesssim\ssf\varphi_0(r)\ssf.
\end{equation*}
\item[(ii)]
Assume that \,$|t|\!\ge\ssb2$\ssf.
\vspace{1mm}
\begin{itemize}
\item[(a)]
If \,$0\ssb\le\ssb r\ssb\le\ssb\frac{|t|}2$\ssf, then
\begin{equation*}
|\ssf w_{\ssf t}^{\ssf0}(r)|\ssf
\lesssim\ssf|t|^{\ssf\tau-3}\,\varphi_0(r)\ssf.
\end{equation*}
\item[(b)]
If \,$r\ssb\ge\ssb\frac{|t|}2$\ssf, then
\begin{equation*}
|\ssf w_{\ssf t}^{\ssf0}(r)|\ssf\lesssim\ssf
(\ssf1\!+\ssb|\ssf r\!-\!|t|\ssf|\ssf)^{\ssf\tau-2}\,e^{-\frac{Q}2\ssf r}\ssf.
\end{equation*}
\end{itemize}
\end{itemize}
\end{theorem}

\begin{proof}
Recall that
\begin{equation}\label{wt0}
w_{\ssf t}^{\ssf0}(r)=\const\int_{\,0}^{\ssf2}\!d\lambda\,
\chi_0(\lambda)\,|\mathbf{c}\hspace{.1mm}(\lambda)|^{-2}\,
\lambda^{-\tau}\,\bigl(\lambda^2\!+\ssb{\textstyle\frac{\widetilde{Q}^2}4}\bigr)^{\!\frac{\tau-\sigma}2}\,
\varphi_\lambda(r)\,e^{\ssf i\ssf t\ssf\lambda}\,.
\end{equation}
By symmetry we may assume that \ssf$t\!>\!0$\ssf.

\noindent
(i) It follows from the estimates \eqref{philphi0} and \eqref{HCest} that
\begin{equation*}
|\ssf w_{\ssf t}^{\ssf0}(r)|\,
\lesssim\int_{\,0}^{\ssf2}d\lambda\,\lambda^{2-\tau}\,\varphi_0(r)\,
\lesssim\,\varphi_0(r)\,.
\end{equation*}

\noindent
(ii) We prove first (a) by substituting the representation \eqref{radialisationformula} of $\varphi_{\lambda}$ in \eqref{wt0}. Specifically, 
$$
w_t^0(r)=\const \int_{\partial B(\mathfrak{s})} d\sigma \,a(r\sigma)^{Q/2} \int_0^2 d\lambda\,
\chi_0(\lambda)\,b(\lambda)\,e^{i\{t-\log a(r\sigma)  \}\lambda} \,,
$$
where $b(\lambda)=|\mathbf{c}\hspace{.1mm}(\lambda)|^{-2}\,
\lambda^{-\tau}\,\bigl(\lambda^2\!+\ssb{\textstyle\frac{\widetilde{Q}^2}4}\bigr)^{\!\frac{\tau-\sigma}2}\,$ and $a(r\sigma)$ is the $A$-component of the point $r\sigma$ defined in Section \ref{Notation}. According to estimate \eqref{loga} and to Lemma A.1 in Appendix A, the inner integral is bounded above by 
$$
\{t-\log a(r\sigma) \}^{\tau-3}\leq (t-r)^{\tau-3}\asymp t^{\tau-3}\,\qquad\forall\sigma\in\partial B(\mathfrak{s})\,.
$$
Since $\pi\big[a(\cdot)^{Q/2}\big]=\varphi_0$, we conclude that
\begin{equation*}
\begin{aligned}
w_t^0(r)&\lesssim t^{\tau-3}\,\int_{\partial B(\mathfrak{s})} d\sigma \,a(r\sigma)^{Q/2}\\
&= C\,t^{\tau-3}\,\pi\big[a(\cdot)^{Q/2}\big](r)\\
&=C\, t^{\tau-3}\,\varphi_0(r)\,.
\end{aligned}
\end{equation*}

We prove next (b) by substituting in \eqref{wt0}
the asymptotic expansion (\ref{expPhilambda}) of $\varphi_\lambda$
and by reducing to Fourier analysis on \ssf$\R$\ssf.
Specifically,
\begin{equation}\label{sumK}
w_{\ssf t}^{\ssf0}(r)=  \const\, e^{-\frac{Q}2\,r}\,\sum\nolimits_{\ssf \ell=0}^{+\infty}
e^{-\ssf \ell\ssf r}\,\big\{\ssf I_{\,\ell}^{+,0}(t,r)+I_{\,\ell}^{-,0}(t,r)\ssf\big\}\,,
\end{equation}
where
\begin{equation*}
I_{\,\ell}^{\pm,0}(t,r)\ssf=\ssb\int_{\,0}^{\ssf2}\!d\lambda\,
\chi_0(\lambda)\,b_{\ssf \ell}^\pm(\lambda)\,e^{\,i\ssf(t\ssf\pm\ssf r)\ssf\lambda}
\end{equation*}
and
\begin{equation*}
b_{\ssf \ell}^\pm(\lambda)
=\ssf\mathbf{c}\hspace{.1mm}(\mp\lambda)^{-1}\,\lambda^{-\tau}\,
\bigl(\lambda^2\!+\ssb{\textstyle\frac{\widetilde{Q}^2}4}\bigr)^{\!\frac{\tau-\sigma}2}\,
\Gamma_\ell(\pm\lambda)\,.
\end{equation*}
By applying Lemma A.1
and Lemma \ref{l: derGammaelle},
we obtain
\begin{equation*}
|\ssf I_{\,\ell}^{+,0}(t,r)\ssf|\ssf
\lesssim(\ssf1\!+\ssb \ell\ssf)^{\ssf d}\,(\ssf t\!+\!r\ssf)^{\ssf\tau-2}
\le(\ssf1\!+\ssb \ell\ssf)^{\ssf d}\;r^{\ssf\tau-2}
\end{equation*}
and
\begin{equation*}
|\ssf I_{\,\ell}^{-,0}(t,r)\ssf|\ssf\lesssim
(\ssf1\!+\ssb \ell\ssf)^{\ssf d}\,(\ssf1\ssb+\ssb|\ssf r\!-\!t\ssf|\ssf)^{\ssf\tau-2}\,,
\end{equation*}
where $d$ is the constant which appears in Lemma \ref{l: derGammaelle}. 

We conclude the proof by summing up these estimates in \eqref{sumK}.
\end{proof}

\subsection{Estimate of
\,$\widetilde{w}_{\,t}^{\ssf\infty}\ssb
=\ssb\widetilde{w}_{\,t,\infty}^{\ssf(\sigma,\tau)}$.}
\label{KernelEstimatewtildetinfty}

\begin{theorem}\label{Estimatewtildetinfty}
The following pointwise estimates hold for the kernel
\,$\widetilde{w}_{\,t}^{\ssf\infty}\ssb
=\ssb\widetilde{w}_{\,t,\infty}^{\ssf(\sigma,\tau)}$,
for any fixed \ssf$\tau\!\in\!\R$
and uniformly in \ssf$\sigma\!\in\!\mathbb{C}$
with \ssf$\Re\sigma\ssb=\ssb\frac{n+1}2:$
\begin{itemize}
\item[(i)]
Assume that \,$0\!<\!|t|\!\le\!2$\ssf.
\begin{itemize}
\item[(a)]
\,If \,$0\!\le\!r\!\le\!3$\ssf, then
\;$|\,\widetilde{w}_{\,t}^{\ssf\infty}(r)\ssf|\ssf
\lesssim\, 
\;|t|^{-\frac{n-1}2}\,.
 $
\item[(b)]
\,If \,$r\!\ge\!3$\ssf, then
\,$\widetilde{w}_{\,t}^{\ssf\infty}(r)
=\mathrm{O}\bigl(\ssf r^{-\infty}\,e^{-\frac{Q}2\ssf r}\ssf\bigr)$\ssf.
\end{itemize}
\vspace{1mm}
\item[(ii)]
Assume that \,$|t|\!\ge\!2$\ssf. Then
\begin{equation*}\label{wtildeinftytlarge}
|\,\widetilde{w}_{\,t}^{\ssf\infty}(r)\ssf|\,\lesssim\,
(\ssf1\ssb+\ssb|\ssf r\ssb-\ssb|t|\ssf|\ssf)^{-\infty}\,e^{-\frac{Q}2\,r}
\qquad\forall\;r\!\ge\!0\ssf.
\end{equation*}
\end{itemize} 
\end{theorem}
\vspace{2mm}

\noindent
\emph{Proof of Theorem \ref{Estimatewtildetinfty}.ii.}
Recall that, up to a positive constant,
\begin{equation*}\label{wtildetinfty}
\widetilde{w}_{\,t}^{\ssf\infty}(r)
={\textstyle\frac{e^{\sigma^2}}{\Gamma(\frac{n+1}2-\sigma)}}
\int_{\,1}^{+\infty}\hspace{-1mm}d\lambda\;
\chi_\infty(\lambda)\,|\mathbf{c}\hspace{.1mm}(\lambda)|^{-2}\,\lambda^{-\tau}\,
\bigl(\lambda^2\!+\ssb{\textstyle\frac{\widetilde{Q}^2}4}\bigr)^{\!\frac{\tau-\sigma}2}\,
\varphi_\lambda(r)\,e^{\ssf i\ssf t\ssf\lambda}\,.
\end{equation*}
By symmetry we may assume again that \ssf$t\!>\!0$\ssf.
If \ssf$0\!\le\!r\!\le\!\frac t2$\ssf,
we resume the proof of Theorem \ref{Estimatewt0}.ii.a,
using Lemma A.2 instead of Lemma A.1,
and estimate this way
\begin{equation}\label{estimate1wtildetinfty}
|\,\widetilde{w}_{\,t}^{\ssf\infty}(r)\ssf|\,
\lesssim\,(\ssf t\!-\!r\ssf)^{-\infty}\,\varphi_0(r)\,
\lesssim\,t^{-\infty}\,e^{-\frac{Q}2\ssf r}\,.
\end{equation} 

If \ssf$r\!\ge\!\frac t2$\ssf,
we resume the proof of Theorem \ref{Estimatewt0}.ii.b
and expand this way
\begin{equation}\label{expansionwtildetinfty}
\widetilde{w}_{\,t}^{\ssf\infty}(r)=
{\textstyle\frac{e^{\sigma^2}}{\Gamma(\frac{n+1}2-\sigma)}}\,
e^{-\frac{Q}2r}\,\sum\nolimits_{\ssf \ell=0}^{+\infty}e^{-\ell\ssf r}\ssf
\bigl\{\ssf I_{\,\ell}^{+,\infty}(t,r)+I_{\,\ell}^{-,\infty}(t,r)\ssf\bigr\}\,,
\end{equation}
where
\begin{equation*}
I_{\,\ell}^{\pm,\infty}(t,r)\ssf
=\int_{\,0}^{+\infty}\hspace{-1mm}d\lambda\,\chi_{\infty}(\lambda)\,
b_{\ssf \ell}^\pm(\lambda)\,e^{\ssf i\ssf(t\pm r)\ssf\lambda}
\end{equation*}
and
\begin{equation*}
b_{\ssf \ell}^\pm(\lambda)
=\mathbf{c}\hspace{.1mm}(\mp\lambda)^{-1}\,
\lambda^{-\tau}\,\bigl(\lambda^2\!+\ssb{\textstyle\frac{\widetilde{Q}^2}4}\bigr)^{\!\frac{\tau-\sigma}2}\,
\Gamma_{\ell}(\pm\lambda)\,.
\end{equation*}
It follows from the expression \eqref{cfunction} of the ${\bf{c}}$-function and from Lemma \ref{l: derGammaelle} that $b_{\ell}^\pm$ is a symbol of order 
$$
\nu=\begin{cases}
-1&{\rm{if}\,\,}k=0\,,\\
-2&{\rm{if}\,\,}k\in\mathbb{N}^*\,.
\end{cases}
$$
By Lemma A.2 we obtain that forall $N\in \mathbb{N}^*$, there exists a positive constant $C_N$ such that for every $\ell\in \mathbb{N}$
\begin{equation}\label{Iellepiu}
|I^{+,\infty}_{\ell}(t,r)|\leq C_N\,|\sigma|^{N}\,(1+\ell)^{d}\,(t+r)^{-N}
\leq C_N\,|\sigma|^{N}\,(1+\ell)^{d}\,r^{-N}\,,
\end{equation} 
and for every $\ell\in\mathbb{N}^*$
\begin{equation}\label{Iellemeno}
|I^{-,\infty}_{\ell}(t,r)|\leq C_N\,|\sigma|^{N}\,(1+\ell)^{d}\,(1+|r-t|)^{-N}\,,
\end{equation} 
where $d$ is the constant which appears in Lemma \ref{l: derGammaelle}. To estimate the term $I^{-,\infty}_0$ we apply Lemma A.3. To do so, we establish the asymptotic behavior of the symbol $b_0^{-}(\lambda)$, as $\lambda\rightarrow +\infty$. On one hand, by \eqref{cfunction} we have 
\begin{equation*}
\begin{aligned}
\mathbf{c}(\lambda)^{-1}
&\textstyle=\frac1{\Gamma(\frac{n}2)}\,2^{\ssf -Q+i\ssf2\lambda}\,
\frac{\Gamma(i\lambda+\frac{Q}2)\,
\Gamma(i\lambda+\frac{m}4+\frac12)}{\Gamma(i\,2\ssf\lambda)}\\
&\textstyle
=C(n,m,Q)\,\bigl(\frac{i\lambda+\frac{Q}2}{i\lambda}\bigr)^{i\lambda-\frac12}\,
\bigl(i\lambda+\frac{Q}2\bigr)^{\frac{Q}2}\,
\bigl(\frac{i\lambda+\frac{m}4+\frac12}{i\lambda}\bigr)^{i\lambda}\,
\bigl(i\lambda+\frac{m}4+\frac12\bigr)^{\frac{m}4}\,
\big\{1+\text{O}(\lambda^{-1})\big\}\\
&\textstyle
=C(n,m,Q)\,\lambda^{\frac{Q}2+\frac{m}4}
\bigl\{1+\text{O}(\lambda^{-1})\bigr\}\,,
\end{aligned}
\end{equation*}
according to Stirling's formula
\begin{equation*}\label{Stirling}
\Gamma(\xi)=\sqrt{2\ssf\pi\ssf}\,
\xi^{\ssf\xi-\frac12}\,e^{-\xi}\,
\bigl\{\ssf1\ssb+\text{O}\ssf(\ssf|\xi|^{-1})\ssf\bigr\}\,.
\end{equation*}
On the other hand,
\begin{equation*}
\lambda^{-\tau}(\lambda^2\!+\ssb{\tilde\rho}^{\ssf2})^{\frac{\tau-\sigma}2}
=\lambda^{-\sigma}\,\bigl\{\ssf1\ssb
+\text{O}\ssf(\ssf|\sigma|\ssf\lambda^{-2}\ssf)\bigr\}\,.
\end{equation*}
Since $\frac{Q}2+\frac{m}4-\Re\sigma=-1$ we get
\begin{equation*}
b_{\ssf0}^-(\lambda)
=c_{\ssf0}\,\lambda^{-1-i\ssf\Im\sigma}\ssb+r_{\ssf0}(\lambda)
\quad\text{with}\quad
|\ssf r_{\ssf0}(\lambda)|\le C\,|\sigma|\,\lambda^{-2}\,.
\end{equation*}
As announced, it follows now from Lemma A.3 that
\begin{equation}\label{I0minusinftybis}\textstyle
|\ssf I_{\,0}^{-,\infty}(t,r)\ssf|\ssf\le\,C\,\frac{|\sigma|^2}{|\Im\sigma|}
\qquad\text{if \,}|\ssf r\!-\!t\ssf|\!\le\!1\ssf.
\end{equation}
By combining \eqref{estimate1wtildetinfty},
\eqref{expansionwtildetinfty},
\eqref{Iellepiu}, \eqref{Iellemeno} and \eqref{I0minusinftybis},
we conclude that
\begin{equation*}\textstyle
|\,\widetilde{w}_{\,t}^{\ssf\infty}(r)\ssf|\,
\lesssim\,(\ssf1\ssb+\ssb|\ssf r\!-\ssb t\ssf|\ssf)^{-\infty}\,e^{-\frac{Q}2\,r}
\qquad\forall\;r\!\ge\!\frac t2\,.
\end{equation*}

\noindent 
The estimate of Theorem \ref{Estimatewtildetinfty}.i.a is of local nature and thus similar to the Euclidean case.
For the  sake of completeness, we include a proof in Appendix C.
\vspace{2mm}

\noindent
\emph{Proof of Theorem \ref{Estimatewtildetinfty}.i.b.}
Here \ssf$0\!<\!|t|\!\le2$
\ssf and \ssf$r\!\ge\!3$\ssf.
By symmetry we may assume again that \ssf$t\!>\!0$\ssf. 
Up to positive constants,
the inverse spherical Fourier transform \eqref{ISFT}
can be rewritten in the following way\,:
\begin{equation*}\textstyle
\widetilde{w}_{\,t}^{\ssf\infty}(r)\,
=\,\frac{e^{\ssf\sigma^2}}{\Gamma(\frac{n+1}2-\sigma)}\;
\mathcal{A}^{-1}g_{\ssf t}\ssf(r)\,,
\end{equation*}
where
\begin{equation*}
g_{\ssf t}(r)\ssf=\ssf2\int_{\,1}^{+\infty}\hspace{-1mm}d\lambda\;
\chi_\infty(\lambda)\,\lambda^{-\tau}\,
\bigl(\lambda^2\!+\ssb{\textstyle\frac{\widetilde{Q}^2}4}\bigr)^{\!\frac{\tau-\sigma}2}
\,e^{\ssf i\ssf t\ssf\lambda}\,\cos\lambda\ssf r\,.
\end{equation*}
Let us split up
\,$2\cos\lambda\ssf r\ssb
=\ssb e^{\ssf i\ssf\lambda\ssf r}\!+\ssb e^{-i\ssf\lambda\ssf r}$
and \,$g_{\ssf t}\ssf(r)\!=\ssb g_{\,t}^+(r)\!+\ssb g_{\,t}^-(r)$
\ssf accordingly, so that
\begin{equation*}
g_{\,t}^\pm(r)\ssf=\int_{\,1}^{+\infty}\hspace{-1mm}d\lambda\;
\chi_\infty(\lambda)\,\lambda^{-\tau}\,
\bigl(\lambda^2\!+\ssb{\textstyle\frac{\widetilde{Q}^2}4}\bigr)^{\!\frac{\tau-\sigma}2}
\,e^{\,i\ssf(t\ssf\pm\ssf r)\ssf\lambda}\,.
\end{equation*}
\smallskip
Recall that
the inversion formulae \eqref{inv1} and \eqref{inv2} of the Abel transform
involve the differential operators
\ssf$\mathcal{D}_1\!=\ssb-\ssf\frac1{\sinh r}\frac\partial{\partial r}$ \ssf
and \ssf$\mathcal{D}_2\!=\ssb-\ssf\frac1{\sinh (r/2)}\frac\partial{\partial r}$\ssf.
We shall use the fact that,
for all integers \ssf$p\!\ge\!1$ and \ssf$q\!\ge\!1$\ssf,
\begin{equation}\label{D1D2}
\mathcal{D}_1^p\,\mathcal{D}_2^q
=\sum_{j=1}^{p+q}\sum_{\ell=1}^p
\gamma_{\ell,j}^{\ssf\infty}(r)\,
\bigl({\textstyle\frac\partial{\partial r}}\bigr)^j\,,
\end{equation}
where the coefficients $\gamma_{\ell,j}^{\ssf\infty}(r)$
are linear combinations of products
\begin{equation}
\begin{aligned}
&\bigl({\textstyle\frac1{\sinh r}}\bigr)\times
\bigl({\textstyle\frac\partial{\partial r}}\bigr)^{\ell_2}
\bigl({\textstyle\frac1{\sinh r}}\bigr)
\times\,\cdots\,\times
\bigl({\textstyle\frac\partial{\partial r}}\bigr)^{\ell_m}\\
&\times \bigl({\textstyle\frac\partial{\partial r}}\bigr)^{j_1}
\bigl({\textstyle\frac1{\sinh(r/2)}}\bigr)\times
\cdots\,\bigl({\textstyle\frac\partial{\partial r}}\bigr)^{j_q}
\bigl({\textstyle\frac1{\sinh(r/2)}}\bigr),
\end{aligned}
\end{equation}
with \,$\ell_2\!+{\dots}+ \ell_p\!= p - \ell$ and 
\,$j_1\!+{\dots}+ j_q\!= q -j+ \ell$. \\Since $\frac1{\sinh r} 
=2 \sum_{\ssf h=0}^{+\infty} e^{-(2 h+1)  r}$ is $\text{O}\ssf(e^{-r})$, as well as its derivatives, we deduce that $\gamma_{\,\ell,j}^\infty(r)$ is $\text{O}\ssf(e^{-(p+q/2)\ssf r})$ as $r\!\to\!+\infty$. We shall also use the fact that 
\begin{equation*} 
\bigl({\textstyle\frac\partial{\partial r}}\bigr)^j g_{\,t}^\pm(r) 
={\displaystyle\int_{\,1}^{+\infty}}\hskip-1mm
d\lambda\;\chi_\infty(\lambda)\,
\lambda^{-\tau}\,
\bigl(\lambda^2\!+\ssb{\textstyle\frac{\widetilde{Q}^2}4}\bigr)^{\!\frac{\tau-\sigma}2}\,
(\pm\ssf i\lambda)^{j}\,
e^{\,i\ssf(t\pm r)\ssf\lambda}\,.
\end{equation*}
According to Lemma A.2,
for every $N\hspace{-1mm}\in\!\mathbb{N}^*$\ssb,
there exists \ssf$C_N\hspace{-1mm}\ge\!0$ \ssf such that
\begin{equation}\label{derg}
\bigl|\bigl({\textstyle\frac\partial{\partial r}}\bigr)^j g_{\,t}^\pm(r)\bigr|
\le\ssf C_N\,|\sigma|^N\,(\ssf r\ssb\pm\ssb t\ssf)^{-N}\ssf.
\end{equation}

\noindent$\bullet$
\,\emph{Case 1}\,:
\,Assume that \ssf$k$ \ssf is even.
\smallskip

\noindent
By the formula \eqref{inv1} we obtain that 
$$\textstyle
\ssf\widetilde{w}_{\,t}^{\ssf\infty}(r)
=\ssf\const\,\frac{e^{\ssf\sigma^2}}{\Gamma(\frac{n+1}2-\sigma)}\;
\mathcal{D}_1^{k/2}\,\mathcal{D}_2^{m/2}(g_t^+\!+\ssb g_t^-)(r)\,,
$$
which by \eqref{D1D2} and \eqref{derg} is estimated by
\begin{equation*}\textstyle
|\ssf\widetilde{w}_{\,t}^{\ssf\infty}(r)|\ssf
\leq C_N\,r^{-N}\,e^{-\frac{Q}2\ssf r}\qquad\forall N\in \mathbb{N}^*\,.
\end{equation*}

\noindent$\bullet$
\,\emph{Case 2}\,:
\,Assume that \ssf $k$ \ssf is odd.
\smallskip

\noindent
According to \eqref{D1D2} and \eqref{derg},
for every $N\!\in\!\mathbb{N}^*$,
there exists \ssf$C_N\!\ge\!0$ \ssf such that
\begin{equation*}\textstyle
\bigl|\ssf\mathcal{D}_1^{(k+1)/2}\,\mathcal{D}_2^{m/2}g_{\ssf t}(s)\bigr|\ssf
\leq C_N\,|\sigma|^N\,s^{-N}\,e^{- \frac{Q+1}2  \ssf s}
\qquad\forall\;s\!\ge\!3\,.
\end{equation*}
By estimating
\begin{equation*}\begin{aligned}
&\textstyle
\cosh s\ssb-\ssb\cosh r
=2\ssf\sinh\frac{s+r}2\ssf\sinh\frac{s-r}2
\gtrsim e^{\ssf r}\sinh\frac{s-r}2\,,\\
&\textstyle
\sinh s\lesssim e^{\ssf s}\,,
\hspace{2mm}e^{-\frac{Q}2\ssf s}\le e^{-\frac{Q}2\ssf r}\,,
\hspace{2mm}s^{-N}\le r^{-N}\,,
\end{aligned}\end{equation*}
and performing the change of variables \ssf$s\!=\!r\!+\!u$\ssf,
we deduce that 
\begin{equation*}
\begin{aligned}
|\ssf\widetilde{w}_{\,t}^{\ssf\infty}(r)\ssf|\,
&\lesssim\textstyle\,
\frac{e^{\ssf\sigma^2}}{\Gamma(\frac{n+1}2-\sigma)}\,
{\displaystyle\int_{\,r}^{+\infty}}\hspace{-1mm}ds\;
\frac{\sinh s}{\sqrt{\ssf\cosh s\,-\,\cosh r\ssf}}\;
\bigl|\ssf\mathcal{D}_1^{(k+1)/2}\,\mathcal{D}_2^{m/2}  g_{\ssf t}(s)\ssf\bigr|\\
&\le\,C_N\int_{\,r}^{+\infty}\textstyle\hspace{-1mm}ds\;
\frac{\sinh s}{\sqrt{\ssf\cosh s\,-\cosh r\ssf}}\;s^{-N}\,e^{-\frac{Q+1}2\ssf s}\\
&\le\,C_N\;r^{-N}\,e^{-\frac{Q}2\ssf r}
\int_{\,0}^{+\infty}\textstyle\hspace{-1mm}
\frac{du\vphantom{\big|}}{\sqrt{\ssf\sinh\frac u2\ssf}}\,\\
&\le\,C_N\;r^{-N}\,e^{-\frac{Q}2\ssf r}\,.
\end{aligned}
\end{equation*}
\vspace{-5mm}
\hfill$\square$

\section{Dispersive estimates}
\label{Dispersive}

In this section we obtain $L^{q'}\!\to\! L^q$ estimates for the operator
$D^{-\tau}\ssf\tilde{D}^{\ssf\tau-\sigma}\ssf e^{\,i\,t\ssf D}$,
which will be crucial for our Strichartz estimates in next section.
Let us split up its kernel
\ssf$w_t\!=\ssb w_{\ssf t}^{\ssf0}\!+\ssb w_{\,t}^\infty$
as before. We will handle the contribution of \ssf$w_{\ssf t}^{\ssf0}$,
using the pointwise estimates obtained in Subsection \ref{KernelEstimatewt0}
and the following criterion. 

\begin{lemma}\label{KS}
There exists a positive constant \,$C$ such that,
for every radial measurable function \,$\kappa$ on \,$S$,
for every \,$2\!\le\!q,\tilde{q}\!<\!\infty$ and \ssf$f\!\in\!L^{q'}\ssb(S)$,
\begin{equation*}
\|\ssf f\ssb*\ssb\kappa\,\|_{L^q\vphantom{L^{q'}}}
\le\,C\;\|f\|_{L^{\tilde{q}'}}\,\Bigl\{\ssf\int_{\,0}^{+\infty}\hspace{-1mm}dr\,
V(r)\,\varphi_0(r)^{\ssf\nu}\,|\kappa(r)|^{\ssf \alpha}\,\Bigr\}^{\frac1{\alpha}}\,.
\end{equation*}
where \,$\nu\ssb=\ssb\frac{2\ssf\min\ssf\{q,\ssf\tilde{q}\}}{q\ssf+\ssf\tilde{q}}$, $\alpha=\alpha(q,\tilde q)\ssb=\ssb\frac{q\ssf\tilde{q}}{q\ssf+\ssf\tilde{q}}$ and $V$ denotes the radial density of the measure $\mu$ 
as in \eqref{density}.  
\end{lemma}
\begin{proof}
This estimate is obtained by interpolation
between the following version of the Herz criterion \cite{Her}
for Damek--Ricci spaces obtained in \cite[Theorem 3.3]{ADY}
\begin{equation*}
\|\ssf f\ssb*\ssb\kappa\,\|_{L^2}
\lesssim\;\|f\|_{L^2}\int_{\,0}^{+\infty}\hspace{-1mm}dr\,
V(r)\,\varphi_0(r)\,|\kappa(r)|\,,
\end{equation*}
and the elementary inequalities
\begin{equation*}
\|\ssf f\ssb*\ssb\kappa\,\|_{L^q\vphantom{L^{q'}}}
\le\,\|f\|_{L^1\vphantom{L^{q'}}}\,\|\kappa\|_{L^q\vphantom{L^{q'}}}\,,
\quad
\|\ssf f\ssb*\ssb\kappa\,\|_{L^\infty\vphantom{L^{q'}}}
\le\,\|f\|_{L^{\tilde{q}'}}\,\|\kappa\|_{L^{\tilde{q}}\vphantom{L^{q'}}}\,.
\end{equation*}
\end{proof}

For the second part \ssf$w_{\ssf t}^{\ssf\infty}$,
we resume the Euclidean approach,
which consists in interpolating analytically between 
$L^2\!\to\!L^2$ and $L^1\!\to\!L^\infty$ estimates
for the family of operators 
\begin{equation}\label{AnalyticFamily}\textstyle
\widetilde{W}_{\,t,\infty}^{\ssf(\sigma,\tau)}
=\,{\textstyle\frac{e^{\sigma^2}}{\Gamma(\frac{n+1}2-\sigma)}}\;
\chi_\infty(D)\,D^{-\tau}\,\tilde{D}^{\ssf\tau-\sigma}\,e^{\,i\,t\ssf D}
\end{equation}
in the vertical strip \ssf$0\ssb\le\ssb\Re\sigma\ssb\le\!\frac{n+1}2$\ssf.

\subsection{Small time dispersive estimate}
 
\begin{theorem}\label{dispersive0}
Assume that
\,$0\ssb<\ssb|t|\ssb\le\ssb2$\ssf,
\ssf$2\ssb<\ssb q\ssb<\ssb\infty$\ssf,
\ssf$0\ssb\le\ssb\tau\ssb<\ssb\frac32$
and \,$\sigma\ssb\ge\ssb(n\ssb+\!1)\ssf(\frac12\!-\!\frac1q)$\ssf.
Then, 
\begin{equation*}
\bigl\|\ssf D^{-\tau}\ssf\tilde{D}^{\ssf\tau-\sigma}\ssf e^{\,i\,t\ssf D}
\ssf\bigr\|_{L^{q'}\ssb\to L^q}\lesssim\, 
\,|t|^{-(n-1)(\frac12-\frac1q)}\,.
\end{equation*}
\end{theorem}

\begin{proof}
We divide the proof into two parts,
corresponding to the kernel decomposition
\ssf$w_t\!=\ssb w_{\ssf t}^{\ssf0}\!+\ssb w_{\,t}^\infty$.
By applying Lemma \ref{KS}
and by using the pointwise estimates in Theorem \ref{Estimatewt0}.i,
we obtain on one hand
\begin{equation*}
\begin{aligned}
\bigl\|\ssf f\ssb*\ssb w_{\ssf t}^{\ssf0}\ssf\bigr\|_{L^q}
&\lesssim\,\Bigl\{\ssf\int_{\,0}^{+\infty}\hspace{-1mm}dr\,
V(r)\,\varphi_0(r)\,|\ssf w_{\ssf t}^{\ssf0}(r)|^{\frac q2}
\,\Bigr\}^{\frac2q}\;\|f\|_{L^{q'}}\\
&\lesssim\,\Big\{\ssf\int_{\,0}^{+\infty}\hspace{-1mm}dr\,
(1\!+\ssb r)^{1+\frac q2}\,e^{-\frac{Q}2 \,r\ssf(\frac q2-1)}
\ssf\Bigr\}^{\frac2q}\;\|f\|_{L^{q'}}\\
&\lesssim\;\|f\|_{L^{q'}}
\qquad\forall\;f\!\in\!L^{q'}.
\vphantom{\int_0^1}
\end{aligned}
\end{equation*}
For the second part,
we consider the analytic family \eqref{AnalyticFamily}.
If \ssf$\Re\sigma\ssb=\ssb0$\ssf, then
\begin{equation*}
\|\ssf f\ssb*\ssb\widetilde{w}_{\,t}^{\ssf\infty}\ssf\|_{L^2}
\lesssim\,\|f\|_{L^2}
\qquad\forall\;f\!\in\!L^2.
\end{equation*}
If \ssf$\Re\sigma\ssb=\ssb\frac{n+1}2$,
we deduce
from the pointwise estimates in Theorem \ref{Estimatewtildetinfty}.i
that
\begin{equation*}
\|\ssf f\ssb*\ssb\widetilde{w}_{\,t}^{\ssf\infty}\ssf\|_{L^\infty}
\lesssim\,|t|^{-\frac{n-1}2}\,\|f\|_{L^1}
\qquad\forall\;f\!\in\!L^1.
\end{equation*}
By interpolation we conclude
for \ssf$\sigma\ssb=\ssb(n+1)\bigl(\frac12\!-\!\frac1q\bigr)$
\ssf that
\begin{equation*}
\bigl\|\ssf f\ssb*\ssb w_{\ssf t}^\infty\ssf\|_{L^q\vphantom{L^{q'}}}
\lesssim\,|t|^{-(n-1)(\frac12-\frac1q)} \|f\|_{L^{q'}}
\qquad\forall\;f\!\in\!L^{q'}.
\end{equation*}
\end{proof}

\subsection{Large time dispersive estimate}
 
\begin{theorem}\label{dispersiveinfty}
Assume that
\,$|t|\ssb\ge\ssb2$\ssf,
\ssf$2\ssb<\ssb q\ssb<\ssb\infty$\ssf,
\ssf$0\ssb\le\ssb\tau\ssb<\ssb\frac32$
and \,$\sigma\ssb\ge\ssb(n\ssb+\!1)\ssf(\frac12\!-\!\frac1q)$\ssf.
Then
\begin{equation*}
\bigl\|\ssf D^{-\tau}\ssf\tilde{D}^{\ssf\tau-\sigma}\ssf e^{\,i\,t\ssf D}
\ssf\bigr\|_{L^{q'}\ssb\to L^q}\lesssim\,|t|^{\ssf\tau-3}\,.
\end{equation*}
\end{theorem}

\begin{proof}
We divide the proof into three parts,
corresponding to the kernel decomposition
\begin{equation*}
w_t=\1_{\ssf B\big(0,\frac{|t|}2\big)}\ssf w_{\ssf t}^{\ssf0}
+\1_{\,S\smallsetminus\ssf B\big(0,\frac{|t|}2\big)}\ssf w_{\ssf t}^{\ssf0}
+\ssf w_{\,t}^\infty\ssf.
\end{equation*}

\noindent
\emph{Estimate 1}\,:
By applying Lemma \ref{KS}
and using the pointwise estimates in Theorem \ref{Estimatewt0}.ii.a, 
we obtain
\begin{equation*}
\begin{aligned}
\|\ssf f*\{\1_{\ssf B\big(0,\frac{|t|}2\big)}\ssf w_{\ssf t}^{\ssf0}\ssf\}\,\|_{L^q}
&\lesssim\,\Bigl\{\ssf\int_{\,0}^{\frac{|t|}2}\!dr\,
V(r)\,\varphi_0(r)\,|\ssf w_{\ssf t}^{\ssf0}(r)|^{\frac q2}
\,\Bigr\}^{\frac2q} \;\|f\|_{L^{q'}}\\
&\lesssim\;\underbrace{\Bigl\{\ssf\int_{\,0}^{+\infty}\hspace{-1mm}dr\,
(1\!+\ssb r)^{1+\frac q2}\,e^{-\frac{Q}2\,r\ssf(\frac q2-1)}
\ssf\Bigr\}^{\frac2q}}_{<+\infty}\,
|t|^{\ssf\tau-3}\;\|f\|_{L^{q'}}
\qquad\forall\;f\!\in\!L^{q'}.
\end{aligned}
\end{equation*}

\noindent
\emph{Estimate 2}\,:
By applying Lemma \ref{KS}
and using the pointwise estimates in Theorem \ref{Estimatewt0}.ii.b, 
we obtain
\begin{equation*}
\begin{aligned}
\|\ssf f*\{\1_{\ssf S\smallsetminus\ssf B\big(0,\frac{|t|}2\big)}\ssf
w_{\ssf t}^{\ssf0}\ssf\}\,\|_{L^q}
&\lesssim\,\Bigl\{\ssf\int_{\,\frac{|t|}2}^{+\infty}\hspace{-1mm}dr\,
V(r)\,\varphi_0(r)^{\frac2q}\,|\ssf w_{\ssf t}^{\ssf0}(r)|^{\frac q2}
\,\Bigr\}^{\frac2q}\;\|f\|_{L^{q'}}\\
&\lesssim\,\underbrace{
\Bigl\{\ssf\int_{\,\frac{|t|}2}^{+\infty}\hspace{-1mm}dr\,
r\,e^{-(\frac q2-1)\ssf\frac{Q}2\ssf r}\,\Bigr\}^{\frac2q}
}_{\lesssim\;|t|^{-\infty}}\,
\|f\|_{L^{q'}}
\qquad\forall\;f\!\in\!L^{q'}.
\end{aligned}
\end{equation*}

\noindent
\emph{Estimate 3}\,:
In order to estimate
the \ssf$L^{q'}\hspace{-1mm}\rightarrow\!L^q$ norm
of \ssf$f\ssb\mapsto\ssb f\ssb*\ssb w_{\,t}^{\ssf\infty}$, 
we use interpolation
for the analytic family \eqref{AnalyticFamily}.
If \ssf$\Re\sigma\ssb=\ssb0\ssf$, then
\begin{equation*}
\|\ssf f*\widetilde{w}_{\,t}^{\ssf\infty}\ssf\|_{L^2}
\lesssim\,\|f\|_{L^2}
\qquad\forall\;f\!\in\!L^2.
\end{equation*}
If \ssf$\Re\sigma\ssb=\ssb\frac{n+1}2$,
we deduce from Theorem \ref{Estimatewtildetinfty}.ii that
\begin{equation*}
\|\ssf f*\widetilde{w}_{\,t}^{\ssf\infty}\ssf\|_{L^\infty}
\lesssim\,|t|^{-\infty} \,\|f\|_{L^1}
\qquad\forall\;f\!\in\!L^1.
\end{equation*}
By interpolation we conclude for
\ssf$\sigma\ssb=\ssb(n\!+\!1)\bigl(\frac12\!-\!\frac1q\bigr)$
that
\begin{equation*}
\bigl\|\ssf f\ssb*\ssb w_{\ssf t}^\infty\ssf\|_{L^q\vphantom{L^{q'}}}
\lesssim\,|t|^{-\infty} \,\|f\|_{L^{q'}}
\qquad\forall\;f\!\in\!L^{q'}.
\end{equation*}
\end{proof}

By taking \ssf$\tau\!=\!1$
\ssf in Theorems \ref{dispersive0} and \ref{dispersiveinfty},
we obtain in particular the following dispersive estimates.

\begin{corollary}\label{DispersiveGlobal}
Let \,$2\!<\!q\!<\!\infty$
and \,$\sigma\!\ge\!(n\!+\!1)\bigl(\frac12\!-\!\frac1q\bigr)$.
Then
\begin{equation*}\textstyle
\|\,\tilde{D}^{-\sigma+1}\,\frac{e^{\,i\ssf t\ssf D}}D\,\|_{L^{q'}\!\to L^q}
\lesssim\,\begin{cases}
\;|t|^{-(n-1)(\frac12-\frac1q)}
&\text{if \;}0\!<\!|t|\!\le\!2\ssf,\\
\;|t|^{-2}
&\text{if \;}|t|\!\ge\!2\ssf\,.
\end{cases}
\end{equation*}
\end{corollary}



\section{Strichartz estimates}
\label{Strichartz}

Consider the inhomogeneous linear wave equation on \ssf$S$\,:
\begin{equation}\label{IP}
\begin{cases}
&\partial_{\ssf t}^{\ssf2}u(t,x)-\bigl(\Delta_S\ssb+\ssb\frac{Q^2}4\bigr)\ssf u(t,x)
=F(t,x)\\
&u(0,x)=f(x)\\
&\partial_{\ssf t}|_{t=0}\,u(t,x)=g(x)\,,
\end{cases}
\end{equation}
whose solution is given by Duhamel's formula\,:
\begin{equation*}\textstyle
u(t,x)=(\cos t\ssf D_x)\ssf f\ssf(x)
+\frac{\sin t\ssf D_x}{D_x}\ssf g\ssf(x)
+{\displaystyle\int_{\,0}^{\ssf t}}ds\,
\frac{\sin\ssf(t-s)\ssf D_x}{D_x}\ssf F(s,x)\,.
\end{equation*}

\begin{definition}\label{admissibility}
A couple $(p,q)$ is called {\emph{admissible}}
if $\bigl(\frac1p,\frac1q\bigr)$ belongs to the triangle
\begin{equation}\label{triangle}\textstyle
T_n=\bigl\{\ssf
\bigl(\frac1p,\frac1q\bigr)\!\in\!\bigl(0,\frac12\bigr]\!\times\!\bigl(0,\frac12\bigr)
\ssb\bigm|\frac2p\!+\!\frac{n-1}q\!\ge\!\frac{n-1}2\,\bigr\}\,.
\end{equation}
\end{definition}

From the dispersive estimates obtained above
and by arguing as in the proof of Theorem \cite[6.3]{APV2} we 
obtain the following result. 

\begin{theorem}\label{StrichartzEstimates}
Let \,$(p,q)$ and \,$(\tilde p, \tilde q)$ be two admissible couples.
Then the following Strichartz estimate holds
for solutions to the Cauchy problem \eqref{IP}\,{\rm :}
\begin{equation}\label{Str1}
\|u\|_{L^p(\R\ssf;\ssf L^q)\vphantom{\big|}}
\lesssim\;\|f\|_{H^{\sigma-\frac12,\frac12}}
+\,\|g\|_{H^{\sigma-\frac12,-\frac12}}
+\,\|F\|_{L^{\tilde{p}'}\ssb\bigl(\R\ssf;\ssf
H_{\tilde{q}'}^{\sigma+\tilde{\sigma}-1}\bigr)}\ssf,
\end{equation}
where \,$\sigma\ssb\ge\ssb\frac{(n+1)}2\ssf\big(\frac12\!-\!\frac1q\big)$
and  \,$\tilde{\sigma}\ssb\ge\ssb
\frac{(n+1)}2\ssf\big(\frac12\!-\!\frac1{\tilde{q}}\big)$.
Moreover,
\begin{equation}\label{Str2}\begin{aligned}
&\|u\|_{L^{\infty}\bigl(\R\ssf;\ssf
H^{\sigma-\frac12,\frac12}\bigr)}
+\|\ssf\partial_{\ssf t}\ssf u\ssf\|_{L^{\infty}\bigl(\R\ssf;\ssf
H^{\sigma-\frac12,-\frac12}\bigr)}\\
&\lesssim\,\|f\|_{H^{\sigma-\frac12,\frac12}\vphantom{\big|}}
+\,\|g\|_{H^{\sigma-\frac12,-\frac12}\vphantom{\big|}}
+\,\|F\|_{L^{\tilde{p}'}\ssb\bigl(\R\ssf;\ssf
H_{\tilde{q}'}^{\sigma+\tilde{\sigma}-1}\bigr)}\,.
\end{aligned}\end{equation}
\end{theorem}

\begin{remark}\label{StrichartzI}
Observe that, in the statement of Theorem \ref{StrichartzEstimates},
we may replace \,$\R$ by any time interval \,$I$ containing \,$0$\ssf.
\end{remark}

\section{GWP results for the NLW equation on $S$}
\label{GWP}
 
We apply Strichartz estimates
for the inhomogeneous linear Cauchy problem associated with the wave equation
to prove global well--posedness results for the following nonlinear  Cauchy problem
\begin{equation}\label{NLWhyperbolic}
\begin{cases}
& \partial_{\,t}^{\ssf 2} u(t,x) -\bigl(\Delta_S\ssb+\ssb\frac{Q^2}4\bigr)\,u(t,x) = F(u(t,x))\,\\
& u(0,x) = f(x)\,\\
& \partial_t|_{t=0}\,u(t,x) = g(x)\,,
\end{cases}
\end{equation}
with a power--like nonlinearity \ssf$F(u)$.
By this we mean that
\begin{equation}\label{power}
|F(u)|\le C\,|u|^\gamma
\quad\text{and}\quad
|\ssf F(u)\ssb-\ssb F(v)\ssf|\ssf\le\ssf
C\,(\ssf|u|^{\gamma-1}\ssb+\ssb|v|^{\gamma-1}\ssf)\,|\ssf u\ssb-\ssb v\ssf|
\end{equation}
for some \ssf$C\!\ge\!0$ \ssf and \ssf$\gamma\!>\!1$\ssf.
Let us recall the definition of global well--posedness.

\begin{definition}
The Cauchy problem \eqref{NLWhyperbolic} is \ssf{\rm globally well--posed}
in \,$H^{\sigma,\tau}\!\times\ssb H^{\sigma,\tau-1}$
if, for any bounded subset \ssf$B$ of
\,$H^{\sigma,\tau}\!\times\ssb H^{\sigma,\tau-1}$,
there exist a Banach space \ssf$X$,
continuously embedded into
\,$C\ssf(\ssf\mathbb{R}\ssf;H^{s,\tau})\cap
C^1(\ssf\mathbb{R}\ssf;H^{s,\tau-1})$\ssf,
such that
\newline
$\bullet$
\,for any initial data \ssf$(f,g)\!\in\!B$, 
$\eqref{NLWhyperbolic}$ has a unique solution \ssf$u\!\in\!X$;
\newline
$\bullet$
\,the map \ssf$(f,g)\ssb\mapsto\ssb u$
is continuous from \ssf$B$ into \ssf$X$.
\end{definition}

The amount of smoothness \ssf$\sigma$ \ssf
requested for GWP of \eqref{NLWhyperbolic}
in \ssf$H^{\sigma-\frac12,\frac12}\ssb\times\ssb H^{\sigma-\frac12,-\frac12}$
depends on $\gamma$ and is represented in Figure 1 below.
There
\vspace{2mm}

\centerline{$\hfill
\gamma_1\ssb
=\ssb\frac{n\ssf+\ssf3}n\ssb
=\ssb1\!+\ssb\frac3n\ssf,\hfill
\gamma_2\ssb
=\ssb\frac{(n+1)^2}{(n-1)^2+\ssf4}\ssb
=\ssb1\!+\ssb\frac2{\frac{n-1}2+\frac2{n-1}}\ssf,\hfill
\gamma_{\text{conf}}\ssb
=\ssb\frac{n\ssf+\ssf3}{n\ssf-\ssf1}\ssb
=\ssb1\!+\ssb\frac4{n\ssf-\ssf1}\ssf,
\hfill$}
\centerline{$
\gamma_3\ssb
=\ssb\frac{n^2+\ssf5\ssf n\ssf-\ssf2\ssf+\ssf
\sqrt{\ssf n^4+\ssf2\ssf n^3+\ssf21\ssf n^2-\ssf12\ssf n\ssf+\ssf4\ssf}}
{2\ssf n^2\ssf-\ssf2\ssf n}\ssb
=\ssb1\!+\ssb\frac{\sqrt{\ssf4\ssf n\ssf+\ssf(\frac{n-6}2-\frac2{n-1})^2}
\,-\,(\frac{n-6}2-\frac2{n-1})}n\ssf,
$}
\centerline{$\hfill
\gamma_4\ssb
=\ssb\frac{n^2+\ssf2\ssf n\ssf-\ssf5}{n^2-\ssf2\ssf n\ssf-\ssf1}\ssb
=\ssb1\!+\ssb\frac2{\frac{n-1}2-\frac1{n-1}}\ssf,\hfill
\gamma_\infty\ssb
=\min\ssf\{\gamma_3,\gamma_4\}
=\begin{cases}
\,\gamma_3
&\text{if \,}n\ssb=\ssb4\ssf,5\ssf,\\
\,\gamma_4
&\text{if \,}n\ssb\ge\ssb6\ssf,\\
\end{cases}
\hfill$}
\vspace{1mm}

\noindent
and the curves \ssf$C_1$, $C_2$\hspace{.1mm}, $C_3$ are given by
\vspace{2mm}

\noindent
\centerline{$\hfill
C_1(\gamma)\ssb
=\ssb\frac{n\ssf+\ssf1}4\ssf\bigl(\ssf1\ssb-\ssb\frac{n\ssf+\ssf5}
{2\,n\ssf\gamma\ssf-\ssf n\ssf-\ssf1}\ssf\bigr)\ssf,\hfill
C_2(\gamma)\ssb
=\ssb\frac{n\ssf+\ssf1}4\ssb-\ssb\frac1{\gamma\ssf-\ssf1}\ssf,
\hfill
C_3(\gamma)\ssb=\ssb\frac n2\ssb-\ssb\frac2{\gamma\ssf-\ssf1}\ssf.
\hfill$}
\vspace{-2.5mm}

\noindent

\begin{figure}[ht]\label{GWPn}
\begin{center}
\psfrag{sigma}[c]{$\sigma$}
\psfrag{0}[c]{$0$}
\psfrag1[c]{$1$}
\psfrag{1/2}[c]{$\frac12$}
\psfrag{n/2}[c]{$\frac n2$}
\psfrag{gamma}[c]{$\gamma$}
\psfrag{gamma1}[c]{$\gamma_1$}
\psfrag{gamma2}[c]{$\gamma_2$}
\psfrag{gammaconf}[c]{$\gamma_{\text{conf}}$}
\psfrag{gammainfty}[c]{$\gamma_\infty$}
\psfrag{C1}[c]{$C_1$}
\psfrag{C2}[c]{$C_2$}
\psfrag{C3}[c]{$C_3$}
\includegraphics[width=10cm]{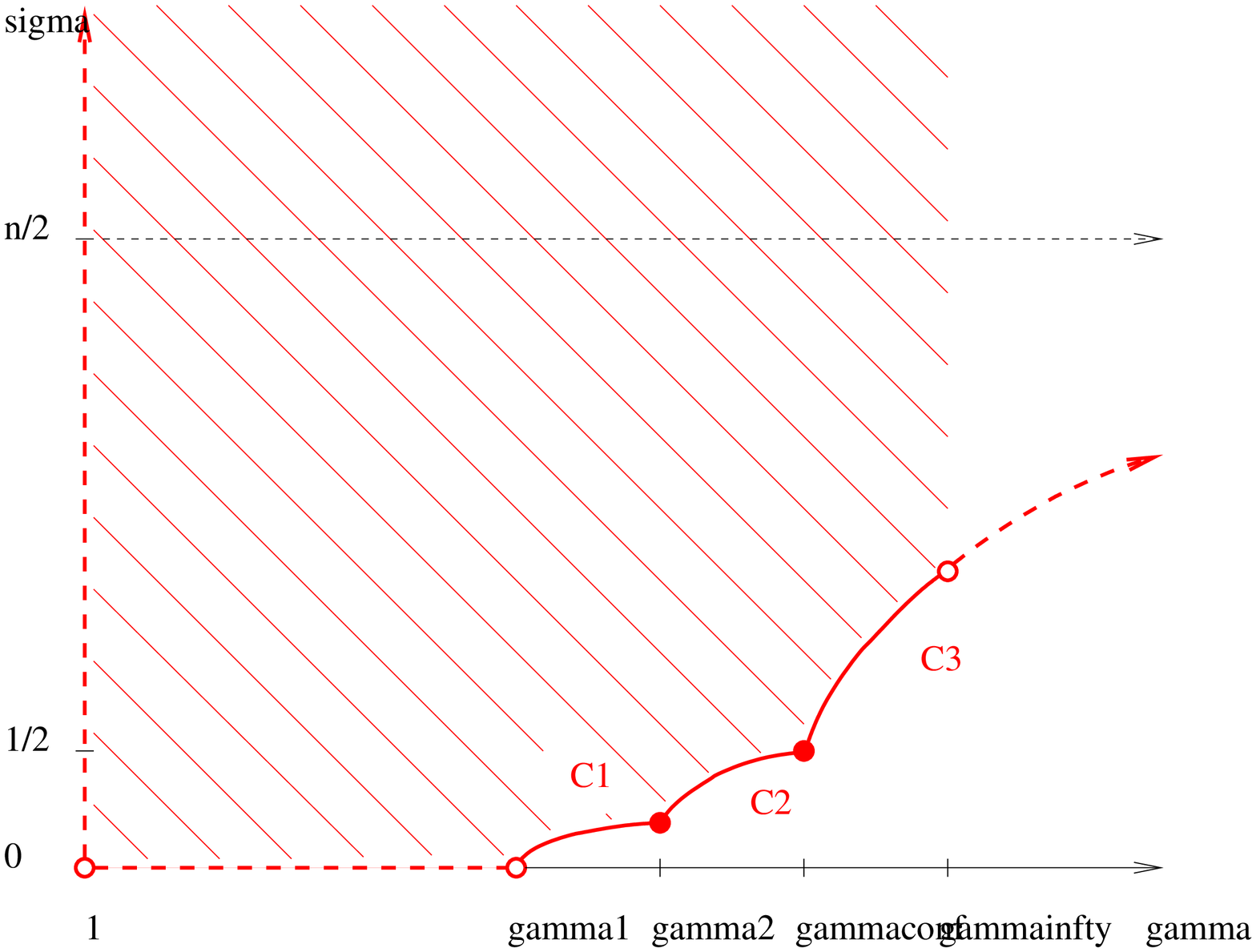}
\end{center}
\caption{Regularity in dimension \ssf$n\ssb\ge\ssb4$}
\end{figure}

\begin{theorem}\label{WPL2}
Assume that \ssf$F(u)$ satisfies \eqref{power}.
Then \eqref{NLWhyperbolic} is globally well--posed for small initial data in
\ssf$H^{\sigma-\frac12,\frac12}\!\times\!H^{\sigma-\frac12,-\frac12}$
in the following cases$\,:$
\begin{itemize}
\item[(A)]
\,$1\!<\!\gamma\!\le\!\gamma_1$
and \,$\sigma\!>\!0$\,;
\item[(B)]
\,$\gamma_1\!<\!\gamma\!\le\!\gamma_2$
and \,$\sigma\!\ge\!C_1(\gamma)$\,{\rm;}
\item[(C)]
\,$\gamma_2\!\le\!\gamma\!\le\!\gamma_{\mathrm{conf}}$
and \,$\sigma\!\ge\!C_2(\gamma)$\,;
\item[(D)]
$\gamma_{\mathrm{conf}}\!\le\!\gamma\!<\!\gamma_{\infty}$
{\rm(\footnote{\,The endpoint \ssf$\gamma\!=\!\gamma_{\infty}$
is excluded in dimension \ssf$n\!=\!4,5$ \ssf
and is actually included in dimension \ssf$n\!\ge\!6\ssf$.})}
and \,$\sigma\!\ge\!C_3(\gamma)$\ssf.
\end{itemize}
More precisely,
for such \,$\gamma$ and \,$\sigma$,
there exists an admissible couple \ssf$(p_0,q_0)$
and, for sufficiently small initial data
\ssf$(f,g)\!\in\!H^{\sigma-\frac12,\frac12}\!\times\!H^{\sigma-\frac12,-\frac12}$,
a unique solution \ssf$u$ to \eqref{NLWhyperbolic} such that
\begin{equation*}
u\ssb\in\ssb
C^1\bigl(\ssf\mathbb{R}\ssf;H^{\sigma-\frac12,\frac12}(S)\bigr)
\cap\ssf L^{p_0}\bigl(\ssf\mathbb{R}\ssf;L^{q_0}(S))
\quad\text{and}\quad
\partial_{\ssf t}u\ssb\in\ssb
C\bigl(\ssf\mathbb{R}\ssf;H^{\sigma-\frac12,-\frac12}(S)\bigr)\ssf.
\end{equation*}
\end{theorem}

\begin{proof}
We apply the standard fixed point method based on Strichartz estimates.
Define \,$u\!=\!\Phi(v)$ \,as the solution to the following linear Cauchy problem
\begin{equation}\begin{cases}
\,\partial_{\,t}^{\ssf2}u(t,x)\ssb-\ssb D_x^{\ssf2}u(t,x)=F(v(t,x))\ssf,\\
\,u(0,x)\ssb=\ssb f(x)\ssf,\\
\,\partial_t|_{t=0}\ssf u(t,x)\ssb=\ssb g(x)\ssf,\\
\end{cases}\end{equation}
which is given by the Duhamel formula
\begin{equation*}\textstyle
u(t,x)
=(\cos t\ssf D_x)\ssf f\ssf(x)
+\frac{\sin t\ssf D_x}{D_x}\,g\ssf(x)
+{\displaystyle\int_{\,0}^{\,t}}\!ds\,\frac{\sin\ssf(t-s)\ssf D_x}{D_x}\ssf F(v(s,x))\,.
\end{equation*}
By Theorem \ref{StrichartzEstimates}
this solution satisfies the Strichartz estimate
\begin{equation*}\label{StrichartzL2v1}
\begin{aligned}
&\|u\|_{L^{\infty}\big(\mathbb{R}\ssf;\ssf H^{\sigma-\frac12,\frac12}\big)}
+\|\partial_tu\|_{L^{\infty}\big(\mathbb{R}\ssf;
\ssf H^{\sigma-\frac12,-\frac12}\big)}
+\|u\|_{L^p(\mathbb{R};\ssf L^q\vphantom{L_t^{\tilde p'}})}\\
&\lesssim\,\|f\|_{H^{\sigma-\frac12,\frac12}}
+\,\|g\|_{H^{\sigma-\frac12,-\frac12}}
+\,\|F(v)\|_{L^{\tilde{p}'}\ssb\bigl(\mathbb{R}\ssf;
\ssf H^{{\sigma+\tilde{\sigma}}-1}_{\tilde{q}'}\bigr)}\,,
\end{aligned}
\end{equation*}
which hold for all admissible couples $(p,q)$, $(\tilde{p},\tilde{q})$
introduced in Definition \ref{admissibility} and for all \ssf$\sigma\!\ge\!\frac{n+1}2\bigl(\frac12\!-\!\frac1q\bigr)$,
$\tilde{\sigma}\!\ge\!\frac{n+1}2\bigl(\frac12\!-\!\frac1{\tilde{q}}\bigr)$. 
According to the  nonlinear assumption \eqref{power},
we estimate the inhomogeneous term as follows\,:
\begin{equation*}
\|F(v)\|_{L^{\tilde p'}\bigl(\mathbb{R}\ssf;
\ssf\,H^{{\sigma+\tilde{\sigma}}-1}_{\tilde q'}\bigr)}
\lesssim \,\|\,|v|^\gamma\|_{L^{\tilde p'}\bigl(\mathbb{R}\ssf;
\ssf H^{{\sigma+\tilde{\sigma}}-1}_{\tilde q'}\bigr)}\ssf.
\end{equation*}
Assuming
\ssf$\sigma\ssb+\ssb\tilde{\sigma}\ssb-\!1\ssb
\le\ssb n\,(\frac1{\tilde{q}'}\!-\!\frac1{\tilde{q}_1'})\ssb
\le\ssb0$\ssf,
we deduce from Sobolev's embedding (Proposition \ref{SET}) that
\begin{equation}\label{StrichartzL2v2}
\begin{aligned}
&\|u\|_{L^{\infty}\bigl(\mathbb{R}\ssf;
\ssf H^{\sigma-\frac12,\frac12}\bigr)}
+\,\|\partial_tu\|_{L^{\infty}\bigl(\mathbb{R}\ssf;
\ssf H^{\sigma-\frac12,-\frac12}\bigr)}
+\,\|u\|_{L^p(\mathbb{R}\ssf;\ssf
L^q\vphantom{L_t^{\tilde p'}})}\\
&\lesssim\,\|f\|_{H^{\sigma-\frac12,\frac12}}
+\,\|g\|_{H^{\sigma-\frac12,-\frac12}}
+\,\|v\|_{L^{\tilde{p}'\gamma}\big( \mathbb{R}\ssf;\ssf
L^{\tilde{q}_1'\ssb\gamma}\big)}^{\,\gamma}\ssf.
\end{aligned}
\end{equation}
In order to remain within the same function space,
we require that \ssf$q\ssb=\ssb\tilde{q}_1'\gamma$ and 
$p=\tilde{p}'\gamma$. It remains for us to check that
the following conditions can be fulfilled simultaneously\,:
\begin{equation}\label{condts}
\begin{cases}
\;\text{(i)}&
p\ssb=\ssb\tilde{p}'\gamma\,,\\
\;\text{(ii)}&
0\ssb<\ssb\frac1{\tilde{q}'}\ssb\le\ssb\frac\gamma q\ssb<\ssb1\,,\\
\;\text{(iii)}&
\frac{n-1}2\ssb-\ssb\frac{n+1}2\ssf
\bigl(\frac1q\ssb+\ssb\frac1{\tilde{q}}\bigr)\ssb
\le\ssb n\ssf\bigl(\frac1{\tilde{q}'}\ssb-\ssb\frac\gamma q\bigr)\,,\\
\;\text{(iv)}&
\frac2p\ssb+\ssb\frac{n-1}q\ssb\ge\ssb\frac{n-1}2\,,\\
\;\text{(v)}&
\frac2{\tilde{p}}\ssb+\ssb\frac{n-1}{\tilde{q}}\ssb\ge\ssb\frac{n-1}2\,,\\
\;\text{(vi)}&
\bigl(\frac1p,\frac1q\bigr)\ssb
\in\ssb\bigl(0,\frac12\bigr]\ssb
\times\ssb\bigl[\frac{n-3}{2\ssf(n-1)},\frac12\bigr)\,,\\ 
\;\text{(vii)}&
\bigl(\frac1{\tilde{p}},\frac1{\tilde{q}}\bigr)\ssb
\in\ssb\bigl(0,\frac12\bigr]\ssb
\times\ssb\bigl[\frac{n-3}{2\ssf(n-1)},\frac12\bigr)\,. 
\end{cases}
\end{equation}
Suppose indeed that there exist indices
\ssf$p,q,\tilde{p},\tilde{q}$ \ssf
satisfying all conditions in (\ref{condts}).
Then \eqref{StrichartzL2v2} shows that
\ssf$\Phi$ \ssf maps \ssf$X$ \ssf into itself,
where \ssf$X$ \ssf denotes the Banach space
\begin{equation*}\begin{aligned}
X=\bigl\{\,u\,\big|\;
&u\ssb\in\ssb
C\ssf(\ssf\mathbb{R}\ssf;H^{\sigma-\frac12,\frac12}(S))
\ssf\cap\ssf L^p(\ssf \mathbb{R}\ssf;L^q(S))\,,\\
&\partial_{\ssf t\ssf}u\ssb\in\ssb
C\ssf(\mathbb{R}\ssf;H^{\sigma-\frac12,-\frac12}(S))\,\bigr\}\,,
\end{aligned}\end{equation*}
equipped with the norm
\begin{equation*}
\|u\|_{X\vphantom{H^{\frac12}}}
=\,\|u\|_{L^\infty\bigl(\ssf \mathbb{R}\ssf;\ssf
H^{\sigma-\frac12,\frac12}\bigr)}
+\,\|\partial_{\ssf t\ssf}u\|_{L^\infty\bigl(\ssf \mathbb{R}\ssf;\ssf
H^{\sigma-\frac12,-\frac12}\bigr)}
+\,\|u\|_{L^p\bigl(\ssf \mathbb{R}\ssf;\ssf L^q\bigr)}\,,
\end{equation*}
Moreover we shall show that
\ssf$\Phi$ \ssf is a contraction on the ball
\begin{equation*}
X_{\varepsilon}=\{\,u\!\in\!X\mid\|u\|_X\ssb\le\varepsilon\,\}\,,
\end{equation*}
provided \ssf$\varepsilon\!>\!0$ \ssf and \ssf
$\|f\|_{H^{\sigma-\frac12,\frac12}}\!+\|g\|_{H^{\sigma-\frac12,-\frac12}}$
are sufficiently small.
Let \ssf$v,\tilde{v}\!\in\!X$
\ssf and \ssf$u\!=\!\Phi(v)$\ssf,
\ssf $\tilde{u}\!=\!\Phi(\tilde{v})$\ssf.
By arguing as above and using H\"older's inequality,
we have
\begin{equation}\label{contractionL2}
\begin{aligned}
\|\,u\ssb-\ssb\tilde{u}\,\|_{X\vphantom{L_t^{\tilde{p}'}}}
&\le\,C\;\|\ssf F(v)\ssb-\ssb F(\tilde{v})\ssf
\|_{L^{\tilde{p}'}\bigl(\ssf\mathbb{R}\ssf;
\ssf H_{\tilde{q}'}^{\sigma+\tilde{\sigma}-1}\big)}\\
&\le\,C\;\bigl\|\ssf
\{\ssf|v|^{\gamma-1}\!+|\tilde{v}|^{\gamma-1}\ssf\}
\,|\ssf v\ssb-\ssb\tilde{v}\ssf|\,
\bigr\|_{L^{\tilde{p}'}\bigl(\ssf\mathbb{R}\ssf;
\ssf L^{\tilde{q}_1'}\bigr)}\\
&\le\,C\;
\bigl\{\ssf\|v\|_{L^p\bigl(\mathbb{R}\ssf;\ssf L^q\bigr)}^{\,\gamma-1}\!
+\|\tilde{v}\|_{L^p\bigl(\mathbb{R}\ssf;\ssf L^q\bigr)}^{\,\gamma-1}\bigr\}\,
\|\ssf v\ssb-\ssb\tilde{v}\ssf\|_{L^p\bigl(\mathbb{R}\ssf;\ssf L^q\bigr)}\\
& \le\,C\; 
\bigl\{\ssf\|v\|_X^{\gamma-1}\!+\|\tilde{v}\|_X^{\gamma-1}\ssf\bigr\}\,
\|\ssf v\ssb-\ssb\tilde{v}\ssf\|_X\,.
\end{aligned}
\end{equation}
If \ssf$\|v\|_X\!\le\ssb\varepsilon$ \ssf
and \ssf$\|\tilde{v}\|_X\!\le\ssb\varepsilon$ \ssf and
\ssf$\|f\|_{H^{\sigma-\frac12,\frac12}}\!
+\|g\|_{H^{\sigma-\frac12,-\frac12}}\!\le\ssb\delta$\ssf,
then \eqref{StrichartzL2v2} yields on one hand
\begin{equation*}
\|u\|_X\ssb\le\ssf C\,\delta+\ssf C\,\varepsilon^{\ssf\gamma}
\quad\text{and}\quad
\|\tilde{u}\|_X\ssb\le\ssf C\,\delta+\ssf C\,\varepsilon^{\ssf\gamma}\,,
\end{equation*}
while \eqref{contractionL2} yields on the other hand
\begin{equation*}
\|\ssf u\ssb-\ssb\tilde{u}\ssf\|_X\ssb
\le2\,C\,\varepsilon^{\gamma-1}\,\|\ssf v\ssb-\ssb\tilde{v}\ssf\|_X\,.
\end{equation*}
Thus, if we choose \ssf$\varepsilon\!>\!0$ \ssf
and \ssf$\delta\!>\!0$ \ssf so small that 
\ssf$C\,\varepsilon^{\gamma-1}\!\le\ssb\frac14$
and \ssf$C\,\delta\ssb\le\ssb\frac34\,\varepsilon$\ssf,
then
\begin{equation*}\textstyle
\|u\|_X\ssb\le\varepsilon\ssf,\;\|\tilde{u}\|_X\ssb\le\varepsilon
\quad\text{and}\quad
\|\ssf u\ssb-\ssb\tilde{u}\ssf\|_X\ssb
\le\frac12\,\|\ssf v\ssb-\ssb\tilde{v}\ssf\|_X,
\end{equation*}
if \ssf$v,\tilde{v}\ssb\in\!X_{\varepsilon}$
and \ssf$u\ssb=\ssb\Phi(v)$\ssf, $\tilde{u}\ssb=\ssb\Phi(\tilde{v})$\ssf.
Hence the map \ssf$\Phi$ \ssf is a contraction
on the complete metric space $X_{\varepsilon}$
and the fixed point theorem allows us to conclude.

Let us eventually prove the existence
of couples $(p,q)$ and $(\tilde{p},\tilde{q})$
satisfying all conditions in \eqref{condts}.
Condition (\ref{condts}.iii) amounts to
\begin{equation}\label{cond0}\textstyle
\frac{2\ssf n\ssf\gamma\ssf-\ssf n\ssf-\ssf1}q+\frac{n-1}{\tilde{q}}\le n\ssb+\ssb1
\quad\mathrm{i.e.}\quad
\frac1{\tilde{q}}\le\frac{n+1}{n-1}\ssb-\ssb\frac{2n\gamma-n-1}{n-1}\frac1q\,.
\end{equation}
By combining \eqref{cond0} with (\ref{condts}.ii)  and (\ref{condts}.vi),
we deduce that
\begin{equation*}\textstyle
\frac{n-3}{2(n-1)}\le\frac1q\le\frac2{(\gamma-1)(n+1)}\,.
\end{equation*}
This implies that \ssf$\gamma\ssb
\le\ssb\widetilde{\gamma}_\infty\!
=\ssb\frac{n^2+2n-7}{(n+1)(n-3)}\ssb
=\ssb1\!+\ssb\frac{4(n-1)}{(n+1)(n-3)}$\ssf.
By combining \eqref{cond0} with (\ref{condts}.vii),
we obtain
\begin{equation*}\textstyle
\frac{n-3}{2(n-1)}\le\frac1{\tilde{q}}\le
\min\big\{\ssf\frac12,\frac{n+1}{n-1}\ssb-\ssb\frac{2n\gamma-n-1}{n-1}\frac1q\ssf\bigr\}\,,
\quad\frac1{\tilde{q}}\ne\frac12\,.
\end{equation*}
By combining (\ref{cond0}) with (\ref{condts}.vii),
we also obtain \ssf$\frac1q\le\frac{n+5}{2(2n\gamma-n-1)}$\ssf.
In summary, the conditions on $q$ reduce to
\begin{equation*}\textstyle
\frac{n-3}{2(n-1)}\le\frac1q\le\min\bigl\{
\frac12,
\frac1\gamma,
\frac2{(\gamma-1)(n+1)},
\frac{n+5}{2(2n\gamma-n-1)}
\bigr\}\,,\quad
\frac1q\ne\frac12,\frac1\gamma\,,
\end{equation*}
or case by case to
\begin{itemize}
\item
$1\ssb<\ssb\gamma\ssb\le\ssb\gamma_1$ \ssf and
\ssf$\frac{n-3}{2(n-1)}\ssb\le\ssb\frac1q\ssb<\ssb\frac12$\ssf,
\item
$\gamma_1\ssb<\ssb\gamma\ssb\le\ssb\gamma_2$
\ssf and \ssf$\frac{n-3}{2(n-1)}\ssb
\le\ssb\frac1q\ssb
\le\ssb\frac{n+5}{2(2n\gamma-n-1)}$\ssf,
\item
$\gamma_2\ssb<\ssb\gamma\ssb\le\ssb\widetilde{\gamma}_{\infty}$ \ssf and
\ssf$\frac{n-3}{2(n-1)}\ssb\le\ssb\frac1q\ssb\le\ssb\frac2{(\gamma-1)(n+1)}$\ssf.
\end{itemize}
Let us turn to the indices \ssf$p$ \ssf and \ssf$\tilde{p}$\ssf.
According to \eqref{condts}, we have
\begin{equation*}\textstyle
\frac{n-1}2\ssf\bigl(\frac12\ssb-\ssb\frac1q\bigr)\ssb
\le\ssb\frac1p\ssb\le\ssb\frac12
\quad\text{and}\quad
\frac{n-1}2\ssf\bigl(\frac12\ssb-\ssb\frac1{\tilde{q}}\bigr)\ssb
\le\ssb\frac1{\tilde{p}}\ssb\le\frac12\,.
\end{equation*}
Since \ssf$\frac1{\tilde{p}}=1\ssb-\ssb\frac\gamma p$\ssf,
we end up with the following conditions on \ssf$p$ \ssf and \ssf$\tilde{p}$\;: 
\begin{equation}\label{condindecesp}\begin{cases}
\;\mathrm{(i)}&\textstyle
\frac{n-1}2\ssf\bigl(\frac12\ssb-\ssb\frac1q\bigr)\ssb
\le\ssb\frac1p\ssb\le\ssb\min\ssf\bigl\{\frac12,
\frac{5-n}{4\ssf\gamma}\ssb
+\ssb\frac{n-1}{2\ssf\gamma\ssf\tilde{q}}\bigr\}\,,\\
\;\mathrm{(ii)}&\textstyle
\frac{n-1}2\ssf\bigl(\frac12\ssb-\ssb\frac1{\tilde{q}}\bigr)\ssb
\le\ssb\frac1{\tilde{p}}\ssb\le\ssb\frac12\,.\\
\end{cases}\end{equation}
There exist indices \ssf$p$ \ssf and \ssf$\tilde{p}$
\ssf which satisfy \eqref{condindecesp} provided that
\ssf$\frac1{\tilde{q}}\ssb
\ge\ssb\frac{\gamma}2\ssb
+\ssb\frac{n-5}{2\ssf(n-1)}\ssb
-\ssb\frac{\gamma}{q}$\ssf.
We thus have to find \ssf$\tilde{q}$ \ssf such that  
\begin{equation}\label{condtildeq}\textstyle
\max\ssf\bigl\{\frac{n-3}{2\ssf(n-1)},
\frac\gamma2\ssb
+\ssb\frac{n-5}{2\ssf(n-1)}\ssb
-\ssb\frac\gamma q\bigr\}
\le\frac1{\tilde{q}}
\le\min\ssf\bigl\{\frac12,
\frac{n+1}{n-1}\ssb
-\ssb\frac{2\ssf n\ssf \gamma\ssf
-\ssf n\ssf-1}{(n-1)\ssf q}\bigr\}\ssf,
\end{equation}
with \ssf$\frac1{\tilde{q}}\ssb\ne\ssb\frac12$\ssf.
This implies that \ssf$q$ \ssf has to satisfy the following conditions\,: 
\begin{equation}\label{condq}\textstyle
\max\ssf\bigl\{\frac{n-3}{2\ssf(n-1)},
\frac12\ssb-\ssb\frac2{\gamma\ssf(n-1)}\bigr\}
\le\frac1q
\le\min\ssf\bigl\{\frac12,\frac1\gamma,
\frac2{(\gamma-1)\ssf(n+1)},
\frac{n+5}{2\ssf(2n\gamma-n-1)},
\frac{n+7-\gamma\ssf(n-1)}{2\ssf(\gamma-1)\ssf(n+1)}\bigr\}\ssf,
\end{equation}
with \ssf$\frac1q\ssb
\ne\ssb\frac12\ssb-\ssb\frac2{\gamma\ssf(n-1)}\ssf,
\,\frac12\ssf,
\,\frac1\gamma\ssf $\ssf. 
The fact that
\ssf$\frac{n-3}{2\ssf(n-1)}\ssb
\leq\ssb\frac{n+7-\gamma\ssf(n-1)}{2\ssf(\gamma-1)\ssf(n+1)}$
\ssf easily implies that
\ssf$\gamma\!\le\!\gamma_4\!<\tilde{\gamma}_\infty$\ssf. 
The fact that
\ssf$\frac12\ssb-\ssb\frac2{\gamma\ssf(n-1)}\ssb
<\ssb\frac{n+7-\gamma\ssf(n-1)}{2\ssf(\gamma-1)\ssf(n+1)}$
\ssf implies that \ssf$\gamma\!<\!\gamma_3$\ssf.
In summary, here are the final conditions on \ssf$q$\ssf,
depending on \ssf$\gamma$
\ssf and possibly on the dimension \ssf$n$\,:
\begin{itemize}
\item[(A)] 
\,$1\ssb<\ssb\gamma\ssb\le\ssb\gamma_1\ssb=\ssb1\ssb+\ssb\frac3n$
\ssf and
\ssf$\frac{n-3}{2\ssf(n-1)}\ssb\le\ssb\frac1q\ssb<\ssb\frac12$\ssf.
\item[(B)]
\,$\gamma_1\ssb<\ssb\gamma\ssb\le\ssb\gamma_2\ssb
=\ssb\frac{(n+1)^2}{n^2-2n+5}$
\ssf and
\ssf$\frac{n-3}{2\ssf(n-1)}\ssb\le\ssb\frac1q\ssb
\le\ssb\frac{n+5}{2\ssf(2n\gamma-n-1)}$\ssf.
\item[(C)]
\,$\gamma_2\ssb<\ssb\gamma\ssb<\ssb\gamma_{\mathrm{conf}}$
\ssf and
\ssf$\frac{n-3}{2\ssf(n-1)}\ssb\le\ssb\frac1q\ssb
\le\ssb\frac2{(\gamma-1)\ssf(n+1)}$
\ssf when \ssf$n\ssb\ge\ssb5$\ssf.\\
When \ssf$n\ssb=\ssb4$\ssf,
we distinguish two subcases\,:
\begin{itemize}
\item[$\bullet$]
\,$\gamma_2\ssb<\ssb\gamma\ssb\le\ssb2$
\ssf and
\ssf$\frac{n-3}{2\ssf(n-1)}\ssb\le\ssb\frac1q\ssb
\le\ssb\frac2{(\gamma-1)\ssf(n+1)}$\ssf,
\item[$\bullet$]
\,$2\ssb<\ssb\gamma\ssb<\ssb\gamma_{\mathrm{conf}}$
\ssf and
\ssf$\frac12\ssb-\ssb\frac2{\gamma\ssf(n-1)}\ssb
<\ssb\frac1q\ssb\le\ssb\frac2{(\gamma-1)\ssf(n+1)}$\ssf.
\end{itemize}
\item[(D)]
\,When \ssf$n\ssb\ge\ssb6$\ssf,
we distinguish two subcases\,:
\begin{itemize}
\item[$\bullet$]
\,$\gamma_{\mathrm{conf}}\ssb\le\ssb\gamma\ssb\le\ssb2$
\ssf and \ssf$\frac{n-3}{2\ssf(n-1)}\ssb\le\ssb\frac1q\ssb
\leq\ssb\frac{n+7-\gamma\ssf(n-1)}{2\ssf(\gamma-1)\ssf(n+1)}$\ssf,
\item[$\bullet$]
\,$2\ssb<\ssb\gamma\ssb\leq \ssb\gamma_4$ \ssf and
\ssf$\frac12\ssb-\ssb\frac2{\gamma\ssf(n-1)}\ssb<\ssb\frac1q\ssb
\leq\ssb\frac{n+7-\gamma\ssf(n-1)}{2\ssf(\gamma-1)\ssf(n+1)}$\ssf. 
\end{itemize}
When \ssf$n\ssb=\ssb5$\ssf,
we replace \ssf$\gamma_4$ \ssf by \ssf$\gamma_3$\ssf and require $\gamma<\gamma_3$.\\
When \ssf$n\ssb=\ssb4$\ssf,
\ssf$\gamma_{\mathrm{conf}}\ssb\le\ssb\gamma\ssb<\ssb\gamma_3$
\ssf and
\ssf$\frac12\ssb-\ssb\frac2{\gamma\ssf(n-1)}\ssb<\ssb\frac1q\ssb\leq \ssb\frac{n+7-\gamma\ssf(n-1)}{2\ssf(\gamma-1)\ssf(n+1)}$\ssf.
\end{itemize}

Let us now examine these cases separately.
\smallskip

\noindent
{\bf Case (A).}
In this case,
we choose successively \ssf$q$ \ssf such that
\begin{equation*}\textstyle
\frac{n-3}{2(n-1)}\le\frac1q<\frac12\,,
\end{equation*}
$\tilde q$ \ssf satisfying \eqref{condtildeq},
and \ssf$p$\ssf, $\tilde{p}$ \ssf satisfying \eqref{condindecesp}. 
Thus, when
\ssf$1\ssb<\ssb\gamma\ssb\le\gamma_1$
\ssf and \ssf$\sigma\ssb>\ssb0$\ssf,
there exists always an admissible couple $(p,q)$
such that all conditions \eqref{condts} are satisfied and
\ssf$\sigma\ssb\ge\ssb\frac{(n+1)}2\ssf(\frac12\ssb-\ssb\frac1q)$\ssf. 
\smallskip

\noindent
{\bf Case (B).}
In this case,
we choose successively \ssf$q$ \ssf such that
\begin{equation*}\textstyle
\frac{n-3}{2(n-1)}  \leq \frac1{q} \leq \frac{n+5}{2 ( 2n\gamma - n - 1)}
\end{equation*}
$p$\ssf, $\tilde{p}$ \ssf satisfying \eqref{condindecesp}, 
and a correspondent $\tilde q$ which satisfies \eqref{condtildeq}.
Thus, when
\ssf$\gamma_1\ssb<\ssb\gamma\ssb\le\ssb\gamma_2$
\ssf and \ssf$\sigma\ssb
\ge\ssb\frac{n+1}4\ssb-\ssb\frac{(n+1)\ssf(n+5)}{4\ssf(2n\gamma-n-1)}$\ssf,
there exists an admissible couple $(p,q)$ such that 
all conditions \eqref{condts} are satisfied and
\ssf$\sigma\ssb\ge\ssb\frac{(n+1)}2\ssf(\frac12\ssb-\ssb\frac1q)$\ssf.
\smallskip

\noindent
{\bf Case (C).}
Assume first that \ssf$n\ssb\ge\ssb5$\ssf.
we choose successively \ssf$q$ \ssf such that
\begin{equation}\label{conditionC1}\textstyle
\frac{n-3}{2\ssf(n-1)}\le\frac1q\le\frac2{(\gamma-1)\ssf(n+1)}\,,
\end{equation}
$\tilde q$ \ssf satisfying \eqref{condtildeq},
and \ssf$p$\ssf, $\tilde{p}$ \ssf satisfying \eqref{condindecesp}. 
 
Assume next that \ssf$n\ssb=\ssb4$\ssf.
If \ssf$\gamma_2\ssb<\ssb\gamma\ssb\le\ssb2$\ssf,
we choose \ssf$q$ \ssf according to \eqref{conditionC1}.
If \ssf$2\ssb<\ssb\gamma\ssb<\ssb\gamma_{\mathrm{conf}}\ssf$,
we replace \eqref{conditionC1} by
\begin{equation*}\textstyle
\frac12\ssb-\ssb\frac2{\gamma\ssf(n-1)}\ssb<\ssb\frac1q\ssb
\le\ssb\frac2{(\gamma-1)\ssf(n+1)}\,.
\end{equation*}
In both cases,
we can choose afterwards \ssf$\tilde{q},p,\tilde{p}$ \ssf
satisfying \eqref{condtildeq} and \eqref{condindecesp}.  

In summary, when
\ssf$\gamma_2\ssb<\ssb\gamma\ssb<\ssb\gamma_{\mathrm{conf}}$
\ssf and \ssf$\sigma\ssb\ge\ssb\frac{n+1}4\ssb-\ssb\frac1{\gamma-1}$\ssf,
there exists always an admissible couple $(p,q)$
such that all conditions \eqref{condts} are satisfied
and \ssf$\sigma\ssb\ge\ssb\frac{(n+1)}2\ssf(\frac12\ssb-\ssb\frac1q)$\ssf.
\smallskip

\noindent
{\bf Case (D).}
Assume first that \ssf$n\ssb\ge\ssb6$\ssf.
If \ssf$\gamma_{\text{conf}}\ssb\le\ssb\gamma\ssb\le\ssb2$\ssf,
we choose successively \ssf$q$ \ssf such that
\begin{equation}\label{choiceD1}\textstyle
\frac{n-3}{2(n-1)}
\le\frac1q
\leq\frac{n+7-\gamma(n-1)}{2(\gamma-1)(n+1)}\,,
\end{equation}
$\tilde q$ \ssf satisfying \eqref{condtildeq},
and \ssf$p$\ssf, $\tilde{p}$ \ssf satisfying \eqref{condindecesp}. 
If \ssf$2\ssb<\ssb\gamma\ssb\leq \ssb\gamma_4$\ssf,
\eqref{choiceD1} is replaced by
\begin{equation}\label{choiceD2}\textstyle
\frac12\ssb-\ssb\frac2{\gamma(n-1)}
<\frac1q
\leq\frac{n+7-\gamma(n-1)}{2(\gamma-1)(n+1)}\,.
\end{equation}

Assume next that \ssf$n\ssb=\ssb5$\ssf.
We choose again \ssf$q$ \ssf according to \eqref{choiceD1}
if \ssf$\gamma_{\text{conf}}\ssb\le\ssb\gamma\ssb\le\ssb2$
\ssf and according to \eqref{choiceD2}
if \ssf$2\ssb<\ssb\gamma\ssb<\ssb\gamma_3$\ssf.
In both cases,
we can choose afterwards \ssf$\tilde{q},p,\tilde{p}$ \ssf
satisfying \eqref{condtildeq} and \eqref{condindecesp}.  

Assume eventually that \ssf$n\ssb=\ssb4$\ssf.
Then we choose \ssf$q$ \ssf according to \eqref{choiceD1}
and \ssf$\tilde{q},p,\tilde{p}$ \ssf
satisfying \eqref{condtildeq} and (\ref{condindecesp}).

In summary, in this case when \ssf$\sigma\ssb\geq\ssb\frac n2\ssb-\ssb\frac2{\gamma-1}$\ssf,
there exists always an admissible couple $(p,q)$ such that 
all conditions \eqref{condts} are satisfied and
\ssf$\sigma\ssb\ge\ssb\frac{n+1}2\ssf(\frac12\ssb-\ssb\frac1q)$\ssf.
\smallskip

This concludes the proof of Theorem \ref{WPL2}.
\end{proof}

\section*{Appendix A}
\label{AppendixA}

In this appendix we collect some lemmata in Fourier analysis on $\R$
which are used for the kernel analysis in Section \ref{Kernel} and in Appendix C. 
These lemmata are proved in \cite[Appendix A]{APV2}. 
\medskip

\noindent
\textbf{Lemma A.1.}\label{LemmaA1}{\textit{
Let \ssf$b$ be a compactly supported homogeneous symbol
on \ssf$\R$ of order \ssf$\nu\!>\!-1\ssf$.
In other words, \ssf$b$ is a smooth function on \,$\mathbb{R}^*$,
whose support is bounded in \ssf$\mathbb{R}$
and which has the following behavior at the origin\,:
\begin{equation*}
\sup_{\lambda\in\mathbb{R}^*}\,|\lambda|^{\,\ell-\nu}\,
|\,\partial_\lambda^{\,\ell}b(\lambda)\ssf|\,<+\infty
\qquad\forall\;\ell\!\in\!\mathbb{N}\,.
\end{equation*}
Then its Fourier transform
\begin{equation*}
k(x)=\int_{\,0}^{+\infty}\hspace{-1mm}d\lambda\,
b(\lambda)\,e^{\ssf i\ssf\lambda\ssf x}
\end{equation*}
is a smooth function on \,$\mathbb{R}$\ssf, 
with the following behavior at infinity:
\begin{equation*}
k(x)=\mathrm{O}\bigl(\ssf|x|^{-\nu-1}\ssf\bigr)
\quad\text{as \,}|x|\!\to\!\infty\,.
\end{equation*}
More precisely, let \ssf$N$ be the smallest integer $>d\!+\!1$\ssf.
Then $\exists\;C\!\ge\!0$\ssf, $\forall\,x\!\in\!\mathbb{R}^*$,
\begin{equation*}
|\ssf k(x)\ssf|\,\le\,C\;|x|^{-\nu-1}\,\sum_{\ell=0}^N\,\sup_{\lambda\in\R^*}\,
(\ssf1\!+\!|\lambda|\ssf)^{\ssf\ell-\nu}\,|\,\partial_\lambda^{\,\ell}b(\lambda)\ssf|\,.
\end{equation*}
}}

\noindent
\textbf{Lemma A.2.}\label{LemmaA2}{\textit{
Let \ssf$b$ be an inhomogeneous symbol on \ssf$\mathbb{R}$
of order \ssf$\nu\!\in\!\mathbb{R}$\ssf.
In other words,
\ssf$b$ is a smooth function on \ssf$\mathbb{R}$
such that
\begin{equation*}
\sup_{\lambda\in\mathbb{R}}\;(\ssf1\!+\!|\lambda|\ssf)^{\ssf\ell-\nu}\,
|\,\partial_\lambda^{\,\ell}b(\lambda)\ssf|\,<\,+\infty
\qquad\forall\;\ell\!\in\!\mathbb{N}\,.
\end{equation*}
Then its Fourier transform
\begin{equation*}
k(x)\ssf={\displaystyle\int_{-\infty}^{+\infty}}\hskip-1mm
d\lambda\,b(\lambda)\,e^{\ssf i\ssf\lambda\ssf x}
\end{equation*}
is a smooth function on \ssf$\mathbb{R}^*$,
which has the following asymptotic behaviors\,:
\begin{itemize}
\item[(i)] At infinity,
\ssf$k(x)\ssb=\ssb\mathrm{O}\bigl(\ssf|x|^{-\infty}\bigr)$\ssf.
More precisely,
for every \ssf$N\!>\ssf \nu\!+\!1$\ssf,
there exists \,$C_N\!\ge\ssb0$ such that,
for every \,$x\!\in\!\mathbb{R}^*$,
\begin{equation*}
|\ssf k(x)\ssf|\le C_N\,|x|^{-N}
\sup_{\lambda\in\R}\;(\ssf1\!+\!|\lambda|\ssf)^{N-\nu}\,
|\,\partial_\lambda^{\ssf N}\ssb b(\lambda)\ssf|\,.
\end{equation*}
\item[(ii)] At the origin,
\begin{equation*}
k(x)=\begin{cases}
\mathrm{O}\ssf(1)
&\text{if \;}\nu\!<\!-1\ssf,\\
\mathrm{O}\ssf(\ssf\log\frac1{|x|})
&\text{if \;}\nu\!=\!-1\ssf,
\vphantom{\frac||}\\
\mathrm{O}\ssf(\ssf|x|^{-\nu-1})
&\text{if \;}\nu\!>\!-1\ssf.\\
\end{cases}
\end{equation*}
More precisely\,:
\\\hskip-1mm$\circ$
\,If \,$\nu\!<\!-1$\ssf,
then there exists \,$C\!\ge\!0$ such that,
for every \,$x\!\in\!\mathbb{R}$\ssf,
\begin{equation*}
|\ssf k(x)\ssf|\le C\,\sup_{\ssf\lambda\in\mathbb{R}}\,
(\ssf1\!+\!|\lambda|\ssf)^{-\nu}\,|\ssf b(\lambda)\ssf|\,.
\end{equation*}
\hskip-1mm$\circ$
\,If \,$\nu\!=\!-1$\ssf,
then there exists \,$C\!\ge\!0$ such that,
for every \,$0\!<\!|x|\!<\!\frac12$\ssf,
\begin{equation*}
|\ssf k(x)\ssf|\ssf\le\ssf C\,\log{\textstyle\frac1{|x|}}\;
\bigl\{\,\sup_{\ssf\lambda\in\mathbb{R}}\;
(\ssf1\!+\!|\lambda|\ssf)\,|\ssf b(\lambda)\ssf|\ssf
+\,\sup_{\lambda\in\mathbb{R}}\;
(1\!+\!|\lambda|)^2\,|\ssf b'(\lambda)\ssf|\,\bigr\}\,.
\end{equation*}
\hskip-1mm$\circ$
\,If \ssf$\nu\!>\!-1$\ssf,
let \ssf$N$ be the smallest integer $>\nu\!+\!1$\ssf.
Then there exists \,$C\!\ge\!0$ \ssf such that,
for every \,$0\!<\!|x|\!<\!1$\ssf,
\begin{equation*}
|\ssf k(x)\ssf|\ssf\le\ssf C\,|x|^{-\nu-1}\,\sum_{\ell=0}^{N}\;
\sup_{\ssf\lambda\in\R}\;(\ssf1\!+\!|\lambda|\ssf)^{\ssf\ell-\nu}\,
|\,\partial_\lambda^{\,\ell}b(\lambda)\ssf|\,.
\end{equation*}
\item[(iii)] Similar estimates hold for the derivatives
\begin{equation*}
\partial_{\ssf x}^{\,\ell}\,k(x)\ssf=
{\displaystyle\int_{-\infty}^{+\infty}}\hskip-1mm
d\lambda\,(i\ssf\lambda)^\ell\,b(\lambda)\,e^{\ssf i\ssf\lambda\ssf x}
\end{equation*}
which correspond to symbols
\,$b_\ell(\lambda)\ssb=\ssb(i\ssf\lambda)^\ell\,a(\lambda)$
of order \ssf$\nu\!+\!\ell$\ssf.
\end{itemize}
}}

\noindent
\textbf{Lemma A.3.}\label{LemmaA3}
{\textit{Assume that
\begin{equation*}\textstyle
b(\lambda)
=\zeta\,\chi_\infty(\lambda)\,\lambda^{-m-1-i\ssf\zeta}
+\ssf f(\lambda)
\end{equation*}
where \ssf$m\!\in\!\mathbb{N}$\ssf,
$\zeta\!\in\!\mathbb{R}$\ssf,
and \ssf$f$ \ssf is a symbol of order \ssf$\nu\!<\!-\ssf m\!-\!1$\ssf.
Then
\begin{equation*}\textstyle
\partial_{\ssf x}^{\ssf m}\ssf k(x)\ssf
={\displaystyle\int_{-\infty}^{+\infty}}\hskip-1mm
d\lambda\,b(\lambda)\,(i\ssf\lambda)^m\,e^{\,i\ssf\lambda\ssf x}
\end{equation*}
is a bounded function at the origin.
More precisely,
there exists \ssf$C\!\ge\!0$ \ssf such that,
for every \ssf$0\!<\!|x|\!<\!\frac12$\ssf,
\begin{equation*}
|\,\partial_{\ssf x}^{\ssf m}\ssf k(x)\ssf|
\le C\,\bigl\{\,1+\zeta^2\ssb+\ssf\sup_{\ssf\lambda\in\mathbb{R}}\,
(\ssf1\!+\!|\lambda|\ssf)^{-\nu}\,|\ssf f(\lambda)\ssf|\,\bigr\}\,.
\end{equation*}
}}
 
\section*{Appendix B}
\label{AppendixB}

In this appendix we collect some properties of the Riesz distributions.
We refer to \cite[ch.\;1, \S\;3 \& ch.\;2, \S\;2]{GC}
or \cite[ch.\;III, \S\;3.2]{Ho1} for more details.
The Riesz distribution $R_{\,z}^{\ssf+}$ is defined by
\begin{equation}\label{Riesz}
\textstyle
\langle\ssf R_{\,z}^{\ssf+},\varphi\ssf\rangle\ssf
=\,\frac1{\Gamma(z)}\,
{\displaystyle\int_{\,0}^{+\infty}}\hspace{-1mm}
d\lambda\;\lambda^{z-1}\,\varphi(\lambda)
\end{equation}
when \ssf$\Re z\!>\!0$\ssf.
It extends to a holomorphic family
\ssf$\{\ssf R_{\,z}^{\ssf+}\}_{z\in\mathbb{C}}$
\ssf of tempered distributions on \ssf$\mathbb{R}$
\ssf which satisfy the following properties\,:
\begin{itemize}

\item[(i)]
\,$\lambda\ssf R_{\,z}^{\ssf+}\ssb=z\ssf R_{\ssf z+1}^{\,+}$
\;$\forall\;z\!\in\ssb\mathbb{C}$\ssf,

\item[(ii)]
\,$(\frac d{d\lambda})\ssf R_{\,z}^{\ssf+}\ssb=R_{\ssf z-1}^{\,+}$
\;$\forall\;z\!\in\ssb\mathbb{C}$\ssf,

\item[(iii)]
\,$R_{\,0}^{\ssf+}\!=\delta_{\ssf0}$
\,and more generally
\,$R_{-m}^{\ssf+}\ssb=(\frac d{d\lambda})^m\ssf\delta_{\ssf0}$
\;$\forall\;m\!\in\!\mathbb{N}$\ssf,

\item[(iv)]
\,$R_{\ssf z+z'}^{\,+}\ssb=R_{\,z}^{\ssf+}\!*\ssb R_{\,z'}^{\ssf+}$
\;$\forall\;z,z'\!\in\ssb\mathbb{C}$\ssf.
\end{itemize}
Hence
\begin{equation*}\textstyle
\langle\ssf R_{\,z}^{\ssf+},\varphi\ssf\rangle\ssf
=\ssf\langle\ssf(\frac d{d\lambda})^m\ssf R_{\ssf z+m}^{\,+}\ssf,
\varphi\ssf\rangle\ssf
=\,\frac{(-1)^m}{\Gamma(z+m)}\,
{\displaystyle\int_{\,0}^{+\infty}}\hspace{-1mm}
d\lambda\;\lambda^{z+m-1}\,
\bigl(\frac d{d\lambda}\bigr)^m\varphi(\lambda)
\end{equation*}
when \ssf$\Re z\!>\!-m$\ssf.
The Riesz distribution
\ssf$R_{\,z}^{\ssf-}\!=\ssb(R_{\,z}^{\ssf+})^{\vee}$
is defined similarly.
Their Fourier transforms are given by
\begin{itemize}

\item[(v)]
\,$\mathcal{F}R_{\,z}^{\ssf\pm}\ssb
=e^{\ssf\pm\ssf i\frac\pi2z}\,
(\ssf x\ssb\pm\ssb i\ssf0\ssf)^{-z}$
\;$\forall\;z\!\in\!\mathbb{C}$\ssf,
\end{itemize}
\vspace{.5mm}

\noindent
where
\begin{equation*}\textstyle
\langle\ssf(x\ssb\pm\ssb i\ssf0)^z,\varphi\ssf\rangle\ssf
=\ssf\lim_{\ssf\varepsilon\searrow0}{\displaystyle\int_{\ssf\R}}
dx\;(x\ssb\pm\ssb i\ssf\varepsilon)^z\,\varphi(x)
\end{equation*}
when \,$\Re z\!>\!-1$ and 
\begin{equation*}\textstyle
(x\ssb\pm\ssb i\ssf0)^z
=\ssf\Gamma(z\!+\!1)\,\{\ssf R_{z+1}^{\,+}\hspace{-.5mm}
+\ssb e^{\ssf\pm\ssf i\ssf\pi z}\ssf R_{z+1}^{\,-}\ssf\}
\end{equation*}
in general
(notice that there are actually no singularities in the last expression).

\section*{Appendix C}
\label{AppendixC}

In this appendix we prove the local kernel estimates
\begin{equation}
|\,\widetilde{w}_{\,t}^{\ssf\infty}(r)\ssf|\ssf
\lesssim\ssf 
\,|t|^{-\frac{n-1}2}
 \end{equation}
stated in Theorem \ref{Estimatewtildetinfty}.i.a
under the assumptions \ssf$0\!<\!|t|\!\le\!2$\ssf, $0\!\le\!r\!\le\!3$
\ssf and $\Re\sigma\!=\!\frac{n+1}2$.
By symmetry, we may assume again that \ssf$t\!>\!0$\ssf.
\smallskip

\noindent$\bullet$
\,\emph{Case 1}\,:
\,Assume that \ssf$r\!\le\!\frac t2$\ssf.
\smallskip

\noindent
By using the representation \eqref{radialisationformula} 
of the spherical functions,
we obtain
\begin{equation}\label{WaveKernel1} 
\widetilde{w}_{\,t}^{\ssf\infty}(r)
={\textstyle\frac{e^{\ssf\sigma^2}}{\Gamma(-\ssf i\Im\sigma)}}\,
\int_{\partial B(\mathfrak{s})} d\sigma\,a^{\frac{Q}2}(r\sigma)
{\displaystyle\int_{\,1}^{\ssf\infty}}\!d\lambda\,
\chi_\infty(\lambda)\,b(\lambda)\,
e^{\,i\ssf\lambda\ssf\{\ssf t- \log a(r\sigma)\}}\,,
\end{equation} 
where
\begin{equation*}\textstyle
b(\lambda)\ssf=\,
|\mathbf{c}\hspace{.1mm}(\lambda)|^{-2}\,\lambda^{-\tau}\,
\bigl(\lambda^2\!+\ssb\frac{\widetilde{Q}^2}4\bigr)^{\!\frac{\tau-\sigma}2}\,,
\end{equation*}
and $a(r\sigma)$ is the $A$-component of the point $r\sigma$. By \eqref{loga} 
\begin{equation*}\textstyle
|\,t\ssb- \log a(r\sigma)\ssf|\ssf\ge\ssf t\ssb-\ssb r\ssf\ge\frac t2
\qquad \forall \sigma\in\partial B(\mathfrak{s})\,,
\end{equation*}
so that according to Lemma A.2 in Appendix A,
since \ssf$\chi_\infty\ssf b$ \ssf is a symbol of order \ssf$\frac{n-3}2$ the inner integral in \eqref{WaveKernel1} is
\begin{equation*}
\text{O}\ssf\bigl(\,|\sigma|^N\,|\,
t\ssb-\ssb  \log a(r\sigma)  |^{-\frac{n-1}2}\ssf\bigr)
=\ssf \text{O}\ssf\bigl(\,|\sigma|^N\,t^{\ssf-\frac{n-1}2}\ssf\bigr)\,,
\end{equation*}
where \ssf$N$ is the smallest integer $>\frac{n-1}2$\ssf.
Hence
\begin{equation*}
|\,\widetilde{w}_{\,t}^{\ssf\infty}(r)\ssf|\,
\lesssim\,t^{\ssf-\frac{n-1}2}\,.
\end{equation*}

\noindent$\bullet$
\,\emph{Case 2}\,:
\,Assume that \ssf$r\!>\!\frac{t}2$\ssf.
\smallskip

\noindent
In this case we estimate \ssf$\widetilde{w}_t(r)$
using the inverse Abel transform.
More precisely,
we apply the inversion formulae \eqref{inv1} and \eqref{inv2}
to the Euclidean Fourier transform
\begin{equation*}\textstyle
\widetilde{g}_{\,t}^{\ssf\infty}(r)\ssf
=\ssf\frac{e^{\ssf\sigma^2}}{\Gamma(-\ssf i\Im\sigma)}\,
{\displaystyle\int_{\,1}^{+\infty}}\hskip-1mm
d\lambda\,\chi_\infty(\lambda)\,|\mathbf{c}\hspace{.1mm}(\lambda)|^{-2}\,\lambda^{-\tau}\,
\bigl(\lambda^2\!+\ssb\frac{\widetilde{Q}^2}4\bigr)^{\!\frac{\tau-\sigma}2}\,
e^{\,i\ssf t\ssf\lambda}\,\cos\lambda\ssf r\,.
\end{equation*}
We shall use the fact that,
for all integers \ssf$p\!\ge\!1$ and \ssf$q\!\ge\!1$\ssf,
\begin{equation}\label{expansion1}
\mathcal{D}_1^p\,\mathcal{D}_2^q=  
\sum_{j=1}^{p+q}\,\sum_{\ell=1}^{\,p}
\gamma_{\ell,j}^{ 0}(r)\,\bigl({\textstyle\frac1r\frac\partial{\partial r}}\bigr)^j\,,
\end{equation}
where the coefficients \ssf$\gamma_{\ell,j}^{0}(r)$ \ssf in \eqref{expansion1}
are smooth functions on \ssf$\R$\ssf, which are linear combinations of products
\begin{equation*}
\begin{aligned} 
&\textstyle
\bigl(\frac r{\sinh r}\bigr)\times
\bigl(\frac1r\frac\partial{\partial r}\bigr)^{\ell_2}
\bigl(\frac r{\sinh r}\bigr)
\times\,\cdots\,\times
\bigl(\frac1r\frac\partial{\partial r}\bigr)^{\ell_p}
\bigl(\frac r{\sinh r}\bigr)\\
&\textstyle\times
\bigl(\frac1r\frac\partial{\partial r}\bigr)^{j_1}
\bigl(\frac r{\sinh(r/2)}\bigr)
\times\,\cdots\,\times
\bigl(\frac1r\frac\partial{\partial r}\bigr)^{j_q}
\bigl(\frac r{\sinh(r/2)}\bigr)
\end{aligned}
\end{equation*}
with \,$\ell_2\!+{\dots}+\ssb\ell_p\!=\ssb p\!-\!\ell$
\,and \,$j_1\!+{\dots}+j_q\!=\ssb q\ssb-\ssb j\ssb+\ssb\ell$\ssf. 
We shall also use the following expansion 
\begin{equation}\label{expansion2} 
\bigl({\textstyle\frac1r\frac\partial{\partial r}}\bigr)^j
={\displaystyle\sum\nolimits_{h=1}^{\,j}}\ssf
\beta_{j,h}\,r^{\ssf h-2j}\,\bigl({\textstyle\frac\partial{\partial r}}\bigr)^h,
\end{equation}
where the coefficients \ssf$\beta_{j,h}$ \ssf in \eqref{expansion2} are constants. 
\smallskip

\noindent$\circ$
\,\emph{Subcase 2.a}\,:
\,Assume that \ssf$k$ \ssf is even.
Then, up to a multiplicative constant,
\begin{equation*}\textstyle
\widetilde{w}_{\,t}^{\ssf\infty}(r)
=\mathcal{D}_1^{k/2}\,\mathcal{D}_2^{m/2}
\widetilde{g}_{\,t}^{\ssf\infty}(r)\,.
\end{equation*}
Consider first
\begin{equation}\label{expression1}\textstyle
\frac{e^{\ssf\sigma^2}}{\Gamma(-\ssf i\Im\sigma)}\,
{\displaystyle\int_{\,1}^{\ssf\frac6r}}
d\lambda\;\chi_\infty(\lambda)\,\lambda^{-\tau}\,
\bigl(\lambda^2\!+\ssb{\textstyle\frac{\widetilde{Q}^2}4}\bigr)^{\!\frac{\tau-\sigma}2}\,
e^{\,i\ssf t\ssf\lambda}\,
\bigl(\ssf\frac1r\frac\partial{\partial r}\bigr)^j
\cos\lambda\ssf r\,.
\end{equation}
Since
\,$\chi_\infty(\lambda)\ssf\lambda^{-\tau}\ssf
 \bigl(\lambda^2\!+\ssb{\textstyle\frac{\widetilde{Q}^2}4}\bigr)^{\!\frac{\tau-\sigma}2}\,
e^{\,i\ssf t\ssf\lambda}
=\text{O}\ssf(\lambda^{-\frac{n+1}2})$
\,according to the assumption \,$\Re\sigma\!= \frac{n+1}2$,
\,and 
\,$\bigl(\frac1r\frac\partial{\partial r}\bigr)^j\ssb\cos\lambda\ssf r
=\text{O}\ssf(\lambda^{2\ssf j\ssf})$
\,by Taylor's formula, the expression \eqref{expression1} is
\begin{equation*}\begin{cases}
\,\text{O}\ssf(1)
&\text{if \,}1\!\le\!j\!<\!\frac {n-1}4\,,\\
\,\text{O}\ssf(\ssf\log\frac1r\ssf)
&\text{if \;}j\!=\!\frac {n-1}4\,,\\
\,\text{O}\ssf(\ssf r^{\ssf \frac{n-1}2-2\ssf j\ssf})
&\text{if \,}\frac{n-1}4\!< j\!\le\!\frac{n-1}2,\\
\end{cases}\end{equation*}
hence \,$\text{O}\ssf(\ssf r^{-\frac{n-1}2}\ssf)$ \,in all cases.
Consider next
\begin{equation}\label{expression2}\textstyle
\frac{e^{\ssf\sigma^2}}{\Gamma(-\ssf i\Im\sigma)}\,
{\displaystyle\int_{\,\frac6r}^{+\infty}}\hskip-1mm
d\lambda\,\lambda^{-\tau}\,
 \bigl(\lambda^2\!+\ssb{\textstyle\frac{\widetilde{Q}^2}4}\bigr)^{\!\frac{\tau-\sigma}2}\,
r^{\ssf h-2\ssf j\ssf}\bigl(\textstyle\frac\partial{\partial r}\bigr)^h
e^{\,i\ssf(t\pm r)\ssf\lambda}\,.
\end{equation}
Since
\,$\bigl(\textstyle\frac\partial{\partial r}\bigr)^h
e^{\,i\ssf(t\pm r)\ssf\lambda}\ssb
=\ssb(\pm\ssf i\lambda)^h\ssf e^{\,i\ssf(t\pm r)\ssf\lambda}$
\,and
\begin{equation*}
\lambda^{-\tau}\,
\bigl(\lambda^2\!+\ssb{\textstyle\frac{\widetilde{Q}^2}4}\bigr)^{\!\frac{\tau-\sigma}2}\,
(\pm\ssf i\lambda)^h\,e^{\,i\ssf(t\pm r)\ssf\lambda}
=\text{O}\ssf(\lambda^{h-\frac{n+1}2})\,,
\end{equation*}
the expression \eqref{expression2} is easily seen to be
\ssf$\text{O}\ssf(\ssf r^{\ssf \frac{n-1}2-2\ssf j\ssf})$
\ssf as long as \ssf$h\!<\!\frac{n-1}2$\ssf. For the remaining case, where \ssf$h\!=\!j\!=\!\frac{n-1}2$\ssf, let us expand
\begin{equation*}\textstyle
\lambda^{-\tau}\,
(\lambda^2\!+\ssb\tilde\rho^{\ssf2})^{\frac{\tau-\sigma}2}
\lambda^{\frac{n-1}2}
=\ssf\lambda^{-1-\ssf i\Im\sigma}\,
\bigl(\ssf1\!+\ssb\frac{\widetilde{Q}^{\ssf2}}{4\lambda^2}
\bigr)^{\ssb\frac{\tau-\sigma}2}\ssb
=\ssf\lambda^{-1-\ssf i\Im\sigma}
+\ssf\text{O}\ssf\bigl(\,|\sigma|\ssf\lambda^{-3}\,\bigr)
\end{equation*}
and split
\begin{equation*}
\int_{\ssf\frac6r}^{+\infty}
=\,\int_{\ssf\frac6r}^{\frac6r+\frac1{|r\pm t|}}
+\,\int_{\frac6r+\frac1{|r\pm t|}}^{+\infty}
\end{equation*}
in \eqref{expression2}.
The previous splitting is meaningful only if \ssf$r\ssb\ne\ssb t$\ssf. 
On one hand, the resulting integrals
\begin{equation}\label{IntegralI}\textstyle
I_\pm=\ssf\frac{e^{\ssf\sigma^2}}{\Gamma(-\ssf i\Im\sigma)}\,
{\displaystyle\int_{\ssf\frac6r}^{\frac6r+\frac1{|r\pm t|}}}\ssb d\lambda\;
\lambda^{-1-\ssf i\Im\sigma}\,e^{\,i\ssf(t\pm r)\ssf\lambda}
\end{equation}
\vspace{-5mm}

\noindent
and
\vspace{-2mm}
\begin{equation}\label{IntegralII}\textstyle
I\!I_\pm=\ssf\frac{e^{\ssf\sigma^2}}{\Gamma(-\ssf i\Im\sigma)}\,
{\displaystyle\int_{\frac6r+\frac1{|r\pm t|}}^{+\infty}}\ssb d\lambda\;
\lambda^{-1-\ssf i\Im\sigma}\,e^{\,i\ssf(t\pm r)\ssf\lambda}
\end{equation}
are uniformly bounded. This is proved by integrations by parts\,:
\vspace{-1mm}
\begin{equation*}\begin{aligned}
\textstyle
I_\pm&\textstyle
=\ssf\frac{e^{\ssf\sigma^2}}{\Gamma(1-\ssf i\Im\sigma)}\ssf
\overbrace{
\lambda^{-\ssf i\Im\sigma}\,e^{\,i\ssf(t\pm r)\ssf\lambda}\,\Big|
_{\ssf\lambda\ssf=\ssf\frac6r}
^{\ssf\lambda\ssf=\ssf\frac6r+\frac1{|r\pm t|}}
}^{\text{O}\ssf(\ssf1\ssf)}\\
&\textstyle
\ssf\mp\,i\,\frac{e^{\ssf\sigma^2}}{\Gamma(1-\ssf i\Im\sigma)}\,
(\ssf r\!\pm\!t\ssf)\ssf\underbrace{
{\displaystyle\int_{\frac1r}^{\frac1r+\frac1{|r\pm t|}}}\!d\lambda\;
\lambda^{-\ssf i\Im\sigma}\,e^{\,i\ssf(t\pm r)\ssf\lambda}
}_{\text{O}\ssf\left(\frac1{|r\pm t|}\right)}\,
=\,\text{O}(1)\ssf,
\end{aligned}\end{equation*}
\vspace{-4mm}

\noindent
while
\begin{equation*}\begin{aligned}
\textstyle
I\!I_\pm&\textstyle
=\ssf\mp\,i\,\frac{e^{\ssf\sigma^2}}{\Gamma(-\ssf i\Im\sigma)}\,
\frac1{r\pm t}\ssf\overbrace{
\lambda^{-1-\ssf i\Im\sigma}\ssf e^{\,i\ssf(t\pm r)\ssf\lambda}\,\Big|
_{\ssf\lambda\ssf=\ssf\frac6r+\frac1{r\pm t}}
^{\ssf\lambda\ssf=\ssf+\infty}
}^{\text{O}\ssf(\ssf |r\ssf\pm\ssf t|\ssf)}\\
&\textstyle
\ssf\mp\,i\,\frac{e^{\ssf\sigma^2}(1+\ssf i\Im\sigma)}{\Gamma(-\ssf i\Im\sigma)}\,
\frac1{r\pm t}\ssf\underbrace{
{\displaystyle\int_{\frac6r+\frac1{|r\pm t|}}^{+\infty}}\!d\lambda\;
\lambda^{-2-\ssf i\Im\sigma}\ssf
e^{\,i\ssf(t\pm r)\ssf\lambda}
}_{\text{O}\ssf(\ssf| r\ssf\pm\ssf t|\ssf)}
=\,\text{O}(1)\ssf.
\end{aligned}\end{equation*}
\vspace{-2mm}

\noindent
Hence the contributions of \eqref{IntegralI} and \eqref{IntegralII}
to \eqref{expression2} are \,$\text{O}\ssf(\ssf r^{-\frac{n-1}2}\ssf)\ssf$.
On the other hand,
the remainder's contribution to \eqref{expression2} is obviously
\ssf$\text{O}\ssf(\ssf r^{\ssf2-\frac{n-1}2}\ssf)\ssf$.
As a conclusion, for all $r>\frac{t}2$ and $r\neq t$, 
\begin{equation*}\textstyle
|\,\widetilde{w}_{\,t}^{\ssf\infty}(r)\ssf|\,
\lesssim\,r^{-\frac{n-1}2}\,\lesssim\,t^{-\frac{n-1}2}\,.
\end{equation*}
If \ssf$r\ssb=\ssb t$\ssf,
the estimates follows as before and is even easier, 
because \ssf$e^{\ssf i\ssf(t-r)\ssf\lambda}\ssb=\ssb1$\ssf.  
\smallskip

\noindent$\circ$
\,\emph{Subcase 2.b}\,:
\,Assume that \ssf$k$ \ssf is odd.
Then, up to a multiplicative constant,
\begin{equation}\label{IAT}\textstyle
\widetilde{w}_{\,t}^{\ssf\infty}(r)\ssf
=\ssf\frac{e^{\ssf\sigma^2}}{\Gamma(-\ssf i\Im\sigma)}\,
{\displaystyle\int_{\,r}^{+\infty}}\hskip-1mm ds\,
\frac{\sinh s}{\sqrt{\cosh s-\cosh r}}\,
\mathcal{D}_1^{(k+1)/2}\,\mathcal{D}_2^{m/2}\,
\widetilde{g}_{\,t}^{\ssf\infty}(s)\,.
\end{equation}
Let us split
\begin{equation}\label{Integral2}
\int_{\ssf r}^{+\infty}\hspace{-1mm}
=\;\int_{\ssf r}^{\,6}\,+\;\int_{\ssf6}^{+\infty}\,.
\end{equation}
The following estimate is obtained
by resuming the proof of Theorem \ref{Estimatewtildetinfty}.i.b
in the odd--dimensional case\,:
\begin{equation*}\textstyle
\bigl| \mathcal{D}_1^{(k+1)/2}\,\mathcal{D}_2^{m/2}\,
\widetilde{g}_{\,t}^{\ssf\infty}(s)\ssf\bigr|
\lesssim\ssf e^{-\frac{Q+1}2\ssf s}
\quad\forall\;s\!\ge\!6\,.
\end{equation*}
Since
\begin{equation*}\textstyle
{\displaystyle\int_{\ssf6}^{+\infty}}\hspace{-1mm}ds\,
\frac{\sinh s}{\sqrt{\cosh s-\cosh r}}\,e^{-\frac{Q+1}2\ssf s}\,
\lesssim{\displaystyle\int_{\,0}^{+\infty}}\hspace{-1.5mm}
\frac{du\vphantom{\big|}}{\sqrt{\ssf\sinh u\ssf}}\,
<+\infty\,,
\end{equation*}
the contribution to \eqref{IAT}
of the second integral in \eqref{Integral2}
is uniformly bounded.
Thus we are left with the contribution of the first integral,
which is a purely local estimate. To do so, we argue as in \cite[Lemma C.1]{APV2} 
and obtain the following result. 
\medskip

\noindent
\textbf{Lemma C.1}\label{LemmaC1}\emph{
Let \ssf$p,q$ be two integers $\ge\!1$
and let \ssf$\lambda\!\ge\!1$\ssf, \ssf$r\!\le\!3$\ssf.
\begin{itemize}
\item[(i)]
Assume that \ssf$\lambda\ssf r\!\le\!6$\ssf.
Then
\begin{equation*}\textstyle
\theta(\lambda,r)\ssf=
{\displaystyle\int_{\,r}^{\,6}}\ssb ds\;
\frac{\sinh s}{\sqrt{\cosh s\,-\,\cosh r}}\;
\mathcal{D}_1^p\,\mathcal{D}_2^q
\cos\lambda\ssf s
\end{equation*}
is \,$\mathrm{O}\ssf(\ssf\lambda^{2p+2q-1-\varepsilon}\,
r^{-\varepsilon}\ssf)$\ssf,
\ssf for every \,$\varepsilon\!>\!0$\ssf.
\item[(ii)]
Assume that $\lambda\ssf r\!\ge\!6$\ssf.
Then
\begin{equation*}\textstyle
\theta^{\ssf\pm}(\lambda,r)\ssf=
{\displaystyle\int_{\,r}^{\,6}}\ssb ds\;
\frac{\sinh s}{\sqrt{\cosh s\,-\,\cosh r}}\;
\mathcal{D}_1^p\,\mathcal{D}_2^q
e^{\ssf\pm\ssf i\ssf\lambda\ssf s}
\end{equation*}
has the following behavior\,:
\begin{equation*}\textstyle
\theta^{\ssf\pm}(\lambda,r)\ssf=\ssf
c_{\ssf\pm}\,\lambda^{p+q-\frac12}\,
(\ssf\sinh r)^{\frac12-p-q}\,
e^{\ssf\pm\ssf i\ssf\lambda\ssf r}
+\mathrm{O}\ssf(\ssf\lambda^{p+q-1}\,r^{-p-q}\ssf)
\end{equation*}
\end{itemize}
where \,$c_{\ssf\pm}$ is a nonzero complex constant.
}

\begin{proof}
We first prove (i). Recall that
\begin{equation*}\textstyle
\mathcal{D}_1^p\,\mathcal{D}_2^q(\cos\lambda\ssf s)\,
=\,\begin{cases}
\,\text{O}\ssf(\lambda^{2p+2q})
&\text{if \,}\lambda\ssf s\!\le\!6\ssf,\\
\,\text{O}\ssf(\lambda^{p+q}\ssf s^{-p-q})
&\text{if \,}\lambda\ssf s\!\ge\!6\ssf,\\
\end{cases}\end{equation*}
hence
\,$\mathcal{D}_1^p\,\mathcal{D}_2^q(\cos\lambda\ssf s)
=\text{O}\ssf(\lambda^{2p+2q-1-\varepsilon}\ssf s^{-1-\varepsilon})$
\,in both cases.
By combining this estimate with
\begin{equation*}
\sinh s\asymp s\,,
\qquad {\rm{and~}} \qquad
\cosh s-\cosh r\asymp s^2\ssb-r^2\ssf,
\end{equation*}
and by performing an elementary change of variables,
we reach our conclusion\,:
\begin{equation*}\textstyle
|\ssf\theta(\lambda,r)|\,
\lesssim\,\lambda^{2p+2q-1-\varepsilon}\,
{\displaystyle\int_{\,r}^{\,6}}\ssb ds\;
s^{-\varepsilon}\,(s^2\!-\ssb r^2)^{-\frac12}\ssf
\le\,\lambda^{2p+2q-1-\varepsilon}\,r^{-\varepsilon}
\underbrace{{\displaystyle\int_{\,1}^{+\infty}}\!ds\;
s^{-\varepsilon}\,(s^2\!-\!1)^{-\frac12}}_{<+\infty}.
\end{equation*}
We next prove (ii). Recall that
\begin{equation*}\textstyle
\mathcal{D}_1^p\,\mathcal{D}_2^q\ssf
(e^{\ssf\pm\ssf i\ssf\lambda\ssf s})
=\frac{(\pm\,i\ssf\lambda)^{p+q}}{(\sinh s)^p(\sinh s/2)^q}\ssf
e^{\ssf\pm\ssf i\ssf\lambda\ssf s}
+\text{O}\ssf(\lambda^{p+q-1}s^{-p-q-1})
\end{equation*}
The remainder's contribution to \ssf$\theta^{\ssf\pm}(\lambda,r)$
\ssf is estimated as above\,:
\begin{equation*}\textstyle
{\displaystyle\int_{\,r}^{\,6}}\ssb ds\;
\frac{\sinh s}{\sqrt{\cosh s\,-\,\cosh r}}\;
\lambda^{p+q-1}\,s^{-p-q-1}\,
\lesssim\,\lambda^{p+q-1}\,
{\displaystyle\int_{\,r}^{\,6}}\ssb ds\;
s^{-p-q}\,(s^2\!-\ssb r^2)^{-\frac12}\,
\lesssim\,\lambda^{p+q-1}\,r^{-p-q}\ssf.
\end{equation*}
In order to handle the contribution of $\frac{(\pm\,i\ssf\lambda)^{p+q}}{(\sinh s)^p(\sinh s/2)^q}\ssf
e^{\ssf\pm\ssf i\ssf\lambda\ssf s}$ we observe that, since $r\leq s\leq 6$, this term is comparable to 
$\frac{(\pm\,i\ssf\lambda)^{p+q}}{(\sinh s)^{p+q}}\ssf
e^{\ssf\pm\ssf i\ssf\lambda\ssf s}$ and we conclude as in 
\cite[Lemma C.1]{APV2}. 
\end{proof}

\noindent
From now on, the discussion of Subcase 2.b is similar to Subcase 2.a.
On one hand, by applying Lemma C.1.i
with \ssf$p\ssb=\!\frac{k+1}2$ and \ssf$q\ssb=\!\frac m2$,
we obtain
\begin{equation*}\textstyle
 \int_{\,1}^{ \frac6r}
d\lambda\;\chi_\infty(\lambda)\,\lambda^{-\tau}\,
\bigl(\lambda^2\!+\ssb{\textstyle\frac{\widetilde{Q}^2}4}\bigr)^{\!\frac{\tau-\sigma}2}\,
e^{\,i\ssf t\ssf\lambda}\;\theta(\lambda,r)\,
=\,\text{O}\ssf(\ssf r^{k+m  -{\rm{Re}}\sigma\,+1   }\ssf)
=\text{O}\big(r^{-\frac{n-1}2}\big)\,.
\end{equation*}
On the other hand, by expanding
\begin{equation*}\textstyle
\lambda^{-\tau}\,
\bigl(\lambda^2\!+\ssb{\textstyle\frac{\widetilde{Q}^2}4}\bigr)^{\!\frac{\tau-\sigma}2}\,
=\ssf\lambda^{-\sigma}\,
\bigl(\,1\ssb+\ssb\frac{\widetilde{Q}^2}{4\lambda^2}\ssf\bigr)^{\frac{\tau-\sigma}2}
=\ssf\lambda^{-\frac{n+1}2-\ssf i\Im\sigma}
+\ssf\text{O}\ssf\bigl(\,|\sigma|\,\lambda^{-\frac{n+1}2-2}\ssf\bigr)
\qquad\forall\;\lambda\!\ge\!2
\end{equation*}
and \,$\theta^{\ssf\pm}(\lambda,r)$ \ssf according to Lemma C.1.ii\ssf,
we have
\begin{equation*}\begin{aligned}
&\textstyle
\frac{e^{\ssf\sigma^2}}{\Gamma(-\ssf i\Im\sigma)}\,
{\displaystyle\int_{\ssf\frac6r}^{+\infty}}\hspace{-1mm}
d\lambda\;\chi_\infty(\lambda)\,\lambda^{-\tau}\,
\bigl(\lambda^2\!+\ssb{\textstyle\frac{\widetilde{Q}^2}4}\bigr)^{\!\frac{\tau-\sigma}2}\,
e^{\,i\ssf t\ssf\lambda}\;\theta^{\ssf\pm}(\lambda,r)\\
&\textstyle
=\,c_{\ssf\pm}\,(\ssf I_\pm\!+\ssb I\!I_\pm\ssf)\,
(\ssf\sinh r)^{\frac{1-n}2}\ssf
+\,\text{O}\bigl(\ssf r^{\frac52-\frac{n}2}\ssf\bigr)\,,
\end{aligned}\end{equation*}
where \,$I_\pm$ and \,$I\!I_\pm$ denote
the integrals \eqref{IntegralI} and \eqref{IntegralII},
which are uniformly bounded and whose sum is equal to
\begin{equation*}\textstyle
\frac{e^{\ssf\sigma^2}}{\Gamma(-\ssf i\Im\sigma)}\,
{\displaystyle\int_{\ssf\frac6r}^{+\infty}}\hspace{-1mm}d\lambda\;
\lambda^{-1-i\Im\sigma}\,e^{\,i\ssf(t\ssf\pm\ssf r)\ssf\lambda}\,.
\end{equation*}
As a conclusion, we obtain again
\begin{equation*}\textstyle
|\,\widetilde{w}_{\,t}^{\ssf\infty}(r)\ssf|\,
\lesssim\,r^{-\frac{n-1}2}\,
\lesssim\,t^{-\frac{n-1}2}\,.
\end{equation*}
\medskip

\noindent
\textbf{Remark C.3.}\emph{
In order to estimate the wave kernel for small time,
we might have used the \textit{Hadamard parametrix\/} \cite[\S\;17.4]{Ho3}
instead of spherical analysis.
}

\end{document}